\documentclass[smallextended]{svjour3}
\usepackage{graphicx,indentfirst}
\usepackage{amscd}
\usepackage[all]{xy}
\usepackage{amsmath}
\usepackage{enumerate}
\usepackage{amsfonts,amssymb}
\usepackage{extarrows}
\usepackage{latexsym}
\usepackage[american]{babel}

\usepackage{times}
\usepackage{tikz}
\usepackage{tikz-cd}
\usepackage{url}
\usepackage{appendix}
\usepackage{verbatim}
\usepackage{hyperref}
\usepackage[affil-it]{authblk}

\usetikzlibrary{matrix,arrows,decorations.pathmorphing}
\usepackage[a4paper,left=2.5cm,right=2.5cm,bottom=3cm,top=3.5cm]{geometry}

\spnewtheorem{ctn}{Caution}{\it}{}

\begin{document}

\title{Twisted complexes on a ringed space as a dg-enhancement of the derived category of perfect complexes}
\author{Zhaoting Wei}
\institute {Department of Mathematics, Indiana University, 831 E 3rd St, Bloomington, IN 47405, USA\\
\email {zhaotwei@indiana.edu}}


\titlerunning{Twisted complexes and dg-enhancement}

\maketitle

\begin{abstract}
In this paper we study the dg-category of twisted complexes on a ringed space and prove that it gives a new dg-enhancement of the derived category of perfect complexes on that space. A twisted complex  is a collection of locally defined sheaves together with the homotopic gluing data. In this paper we construct a dg-functor from twisted complexes to perfect complexes, which turns out to be a dg-enhancement. This new enhancement has the advantage of being completely geometric and it comes directly from the definition of perfect complex. In addition we will talk about some applications and further topics around twisted complexes. 
\end{abstract}


\section*{Acknowledgement}
The author would like to thank Jonathan Block for introducing the author to this topic and for numerous discussions. He also wants to thank Valery Lunts,  Olaf Schn\"{u}rer, Julian Holstein, Ian Shipman and Shilin Yu for very helpful conversations and comments and thank Leo Alonso, Adeel Khan,  Daniel Miller, Denis Nardin,  and Jeremy Rickard for answering questions related to this work. The author also want to thank the referee for his very careful work.

\setcounter{tocdepth}{2}
\tableofcontents

\section{INTRODUCTION}
The derived categories of perfect complexes and pseudo-coherent complexes on ringed topoi were introduced in SGA 6 \cite{berthelot1966seminaire}. They have played an important role in mathematics ever since. Nevertheless  we would like to consider the differential graded (dg)-\emph{enhancements} of these derived categories. More precisely we have the following definition.

\begin{definition}\label{defi: dg-enhancement}
Let $\mathcal{C}$ be a triangulated category. A dg-enhancement of $\mathcal{C}$ is a pair $(\mathcal{B}, \varepsilon)$ where $\mathcal{B}$ is a pre-triangulated dg-category and
$$
\varepsilon: Ho \mathcal{B}\overset{\sim}{\to} \mathcal{C}
$$
is an equivalence of triangulated categories. Here $Ho \mathcal{B}$ is the homotopy category of $\mathcal{B}$.
\end{definition}

For the derived category $D_{\text{perf}}(X)$ on a (quasi-compact and separated) scheme $X$ we have the classical \emph{injective enhancement}, which consists of h-injective objects, see \cite{lunts2014new} Section 3.1. Although very useful, the injective resolution has its drawback that the modules are too "large" and the construction is not geometric. Therefore we are seeking for a new, more geometric dg-enhancement.

In the late 1970's  Toledo and Tong \cite{toledo1978duality} introduced twisted complexes as a way to get their hands on perfect complexes of sheaves on a complex manifold and implicitly they recognized this was a dg-model for the derived category of perfect complexes.  In this paper we prove in all details that twisted complexes form a dg-model for categories of perfect complexes (and more generally pseudo-coherent complexes) of sheaves on a ringed space under some conditions.

Let us first give an informal description to illustrate the idea of twisted complexes. Recall that a complex of sheaves $\mathcal{S}^{\bullet}$ on $X$ is \emph{perfect} if for any point $x\in X$, there exists an open neighborhood $x\in U\subset X$ and a two-side bounded complex of finitely generated locally free sheaves $E^{\bullet}_U$ on $U$ together with a quasi-isomorphism
$$
\theta_U: E^{\bullet}_U \overset{\sim}{\to} \mathcal{S}^{\bullet}|_U.
$$

For two different open subsets $U_i$ and $U_j$ we have two quasi-isomorphisms
$$
\theta_i: E^{\bullet}_{U_i} \overset{\sim}{\to} \mathcal{S}^{\bullet}|_{U_i}
$$
and
$$
\theta_j: E^{\bullet}_{U_j} \overset{\sim}{\to} \mathcal{S}^{\bullet}|_{U_j}.
$$
For simplicity we denote $E^{\bullet}_{U_i}$ by $E^{\bullet}_i$ and $U_i\cap U_j$ by $U_{ij}$. Hence on $U_{ij}$ we have
$$
\begin{tikzcd}[column sep=small]
E^{\bullet}_i|_{U_{ij}}  \arrow {dr}{\sim}[swap]{\theta_i}& &E^{\bullet}_j|_{U_{ij}} \arrow{dl}{\theta_j}[swap]{\sim}\\
& \mathcal{S}^{\bullet}|_{U_{ij}} &
\end{tikzcd}
$$

Since $E^{\bullet}_i$ and $E^{\bullet}_j$ are bounded and locally free, we can refine the open cover if necessary and lift the identity map on $\mathcal{S}^{\bullet}|_{U_{ij}}$ (under some assumptions on $\mathcal{S}^{\bullet}$, see Lemma \ref{lemma: homotopy inverse on good covers} below) to a map $a_{ji}: E^{\bullet}_j\to E^{\bullet}_i$, i.e. the following diagram
$$
\begin{tikzcd}[column sep=small]
E^{\bullet}_i|_{U_{ij}} \arrow{rr}{a_{ji}} \arrow {dr}{\sim}[swap]{\theta_i}& &E^{\bullet}_j|_{U_{ij}} \arrow{dl}{\theta_j}[swap]{\sim}\\
& \mathcal{S}^{\bullet}|_{U_{ij}} &
\end{tikzcd}
$$
commutes \emph{up to homotopy}.

It is expected that the $a_{ji}$'s play the role of transition functions, but we will show that they do not. Consider a third open subset $U_k$ together with $E^{\bullet}_k$ on it. According to the discussion above, the following diagram
$$
\begin{tikzcd}
E^{\bullet}_i|_{U_{ijk}} \arrow{rr}{a_{ji}}  \arrow {ddr}[swap]{a_{ki}}\arrow{dr}{\theta_i}& &E^{\bullet}_j|_{U_{ijk}} \arrow{ddl}{a_{kj}}\arrow{dl}[swap]{\theta_j}\\
&S^{\bullet}|_{U_{ijk}}&\\
& E^{\bullet}_k|_{U_{ijk}}\arrow{u}{\theta_k} &
\end{tikzcd}
$$
commutes  \emph{up to homotopy}. More precisely, we have a degree $-1$ map $a_{kji}: E^{\bullet}_i\to E^{\bullet-1}_k$ on $U_{ijk}$ such that
 $$
 a_{ki}-a_{kj}a_{ji}=[d, a_{kji}]
 $$
 where $d$ is the differential on $E^{\bullet}_i$ and $E^{\bullet}_k$. In other words, the $a_{ji}$'s only satisfy the cocycle condition \emph{up to homotopy}. Hence we cannot simply use them to glue the $E^{\bullet}_i$'s into a complex of  sheaves on $X$. On the other hand we   expect that the homotopy operator $a_{kji}$'s satisfy  compatible relations up to  higher homotopies.

 Toledo and Tong in \cite{toledo1978duality} show that all these compatibility data together satisfy the Maurer-Cartan equation
$$
 \delta a+a\cdot a=0
$$
which will be explained in Section \ref{section: review of twisted complex}. They call a collection $E^{\bullet}_i$ together with such maps $a$'s a \emph{twisted complex} or \emph{twisted cochain}. In this paper we call it \emph{twisted perfect complex} and keep the term twisted complex for a more general concept (For precise definition, see Section \ref{section: review of twisted complex} of this paper). Moreover, O'Brian, Toledo and Tong have proved that every perfect complex has a \emph{twisted resolution}, see  Proposition 1.2.3 in \cite{o1985grothendieck} or Proposition \ref{prop: existence of resolution} in this paper. This result is closely related to the essential-surjectivity   of a dg-enhancement. Nevertheless, they have not attempted to build any equivalence of categories.

In this paper we  construct a \emph{sheafification functor} which is a dg-functor
$$
\mathcal{S}: \text{Tw}_{\text{perf}}(X)\to \text{Qcoh}_{\text{perf}}(X)
$$
where  $\text{Tw}_{\text{perf}}(X)$ denotes the dg-category of twisted perfect  complexes on $X$ and $\text{Qcoh}_{\text{perf}}(X)$ denotes the dg-category of perfect complexes of quasi-coherent sheaves on $X$.

We will prove that the dg-functor $\mathcal{S}$ gives the expected dg-enhancement.

\begin{theorem}\label{thm: dg-enhancement in the introduction}[See Theorem \ref{thm: dg-enhancement} below]
Under reasonable conditions, the sheafification functor  induces an  equivalence  of categories
$$
\mathcal{S}: Ho\text{Tw}_{\text{perf}}(X)\to D_{\text{perf}}(\text{Qcoh}(X)).
$$
\end{theorem}

We would also like to consider perfect  complexes of general $\mathcal{O}_X$-modules rather than quasi-coherent modules. Actually we have
\begin{theorem}\label{thm: dg-enhancement of A-modules in the introduction}[See Theorem \ref{thm: dg-enhancement of A-modules} below]
Under some additional conditions, the sheafification functor  induces an equivalence  of categories
$$
\mathcal{S}: Ho\text{Tw}_{\text{perf}}(X)\to D_{\text{perf}}(X).
$$
\end{theorem}

Here we briefly mention the strategy of the proof. We extend the dg-category of  twisted perfect complexes to a more general dg-category of  \emph{twisted complexes} on $X$ and define a \emph{twisting functor}
$$
\mathcal{T}: \text{Sh}(X)\to \text{Tw}(X)
$$
which is also a dg-functor, where Sh$(X)$ denotes the dg-category of sheaves on $X$ and Tw$(X)$ denotes the dg-category of twisted complexes on $X$. The essential-surjectivity and fully-faithfulness of $\mathcal{S}$ can be achieved by a careful study of the relations between $\mathcal{S}$ and $\mathcal{T}$.

The constructions and proofs are inspired by \cite{block2010duality} Section 4. In fact Block gives a Dolbeault-theoretic dg-enhancement of perfect complexes in \cite{block2010duality} while our construction can be considered as a \v{C}ech-theoretic enhancement.

\begin{remark}\label{remark: relation with Lunts-Schnurer introduction}
In \cite{lunts2014new} Lunts and Schn\"{u}rer introduced another dg-enhancement of the derived category of perfect complexes on a scheme and they call it \v{C}ech enhancement. Conceptually the enhancement in \cite{lunts2014new} is very similar to the twisted complexes in this paper and we will discuss the relations between them in Section \ref{subsection: fully faithfull}, Remark \ref{remark: relation with Lunts-Schnurer} below.
\end{remark}

This paper is organized as follows: In Section \ref{section: review of twisted complex} we give the definition of twisted (perfect) complexes. We show that these dg-categories have a pre-triangulated structure. Moreover we introduce weak equivalences between  twisted complexes.

In Section \ref{section: twisted complex and dg-enhancement} we construct the dg-enhancement. In more details, we construct the sheafification functor $\mathcal{S}$ in  Section \ref{subsection: sheafification functor}. In Section \ref{subsection: perfectness} we prove that the image of a twisted perfect  complex under $\mathcal{S}$ is really a perfect complex. In Section \ref{subsection: essentially surjective} we prove that $\mathcal{S}$ is essentially surjective and in Section \ref{subsection: fully faithfull} we prove that $\mathcal{S}$ is fully faithful. Hence $\mathcal{S}$ gives the dg-enhancement.

In Section \ref{section: applications} we talk about some applications of twisted complexes. In particular we illustrate the application in descent theory.

In Section \ref{section: further topics} we talk about some further topics. In Section \ref{subsection: twisted coherent complexes} we introduce  the \emph{twisted coherent complexes} and prove that they form a dg-enhancement of the derived category of bounded above complexes of coherent sheaves. Actually the proofs are the same as those for twisted perfect complexes.

In Section \ref{subsection: splitting} we make a digression and discuss the degenerate twisted complexes and show how they give splitting of idempotents.

In Section \ref{subsection: quillen adjunction} we outline an alternative  approach to this object:  We wish to put a suitable model structure on  twisted complexes and view $\mathcal{S}$ and $\mathcal{T}$ in terms of \emph{Quillen adjunctions}.

In Appendix \ref{appendix: quasi-coherent, pseudo-coherent complexes and coherent sheaves} we compare coherent complexes and pseudo-coherent complexes. Moreover we study the relation between quasi-coherent modules and general $\mathcal{O}_X$-modules.

To ensure Theorem \ref{thm: dg-enhancement in the introduction} we need that the open cover $\{U_i\}$ of $X$ is fine enough, in Appendix \ref{appendix: good covers} we discuss good covers of a ringed space $X$.

\section{A REVIEW OF TWISTED COMPLEXES}\label{section: review of twisted complex}

\subsection{A quick review of perfect complexes}
Before talking about twisted complexes, we give a quick review of the derived category of perfect complexes and fix the notations in this subsection. For more details see \cite{thomason1990higher} and \cite{stacks-project}.

\begin{definition}\label{defi: perfect complex}
Let $(X,\mathcal{O}_X)$ be a locally ringed space. A complex $\mathcal{S}^{\bullet}$ is \emph{strictly perfect} if $\mathcal{S}^i$ is zero for all but finitely many $i$ and
$\mathcal{S}^i$ is a direct summand of a finite free
$\mathcal{O}_X$-module for all $i$. The second condition is equivalent to that $\mathcal{S}^i$ is a   finite locally free
$\mathcal{O}_X$-module for all $i$.

Moreover  A complex $\mathcal{S}^{\bullet}$ of $\mathcal{O}_X$-modules is \emph{perfect} if for any point $x\in X$, there exists an open neighborhood $U$ of $x$ and a strictly perfect complex  $\mathcal{E}^{\bullet}_{U}$ on $U$ such that the restriction $\mathcal{S}^{\bullet}|_U$ is isomorphic to $\mathcal{E}^{\bullet}_{U}$ in $D(\mathcal{O}_U-\text{mod})$, the derived category of sheaves of $\mathcal{O}_X$-modules on $U$.
\end{definition}

\begin{ctn}
If we did not assume that $X$ is a locally ringed space,
then it may not be true that a direct summand of a finite free
$\mathcal{O}_X$-module is finite locally free. See \cite[\href{http://stacks.math.columbia.edu/tag/08C3}{Tag 08C3}]{stacks-project}.
\end{ctn}

\begin{remark}\label{remark: stronger defi of perfect complex}
In fact, the definition of perfect complex is equivalent to the stronger requirement that for any point $x\in X$, there exists an open neighborhood $U$ of $x$ and a bounded complex of finite rank locally free sheaves  $\mathcal{E}^{\bullet}_{U}$ on $U$ together with a quasi-isomorphism
$$
\mathcal{E}^{\bullet}_{U}\overset{\sim}{\to} \mathcal{S}^{\bullet}|_U.
$$
See \cite[\href{http://stacks.math.columbia.edu/tag/08C3}{Tag 08C3}]{stacks-project} Lemma 20.38.8 for details.
\end{remark}

\begin{remark}
It is obvious that a strictly perfect complex must be perfect. However, on a general ringed space $(X,\mathcal{O}_X)$  perfect complexes are not necessarily strictly perfect. In \cite{lunts2014new} Section 2.3 Lunts and Schn\"{u}rer say that the scheme $X$ satisfies condition GSP if every perfect complex on $X$ is quasi-isomorphic to a strictly perfect complex. It can be proved that any affine scheme or projective scheme or  separated regular Noetherian scheme satisfies condition GSP, see \cite{berthelot1966seminaire} Expos\'{e} II Proposition 2.2.7 and Proposition 2.2.9, or \cite{thomason1990higher} Example 2.1.2 and Proposition 2.3.1.
\end{remark}

We consider the following  categories.

\begin{definition}\label{defi: derived cat of perfect complex}
Let $\text{Sh}(X)$ be the dg-category of complexes of $\mathcal{O}_X$-modules on $X$. Let $\text{Sh}_{\text{perf}}(X)$ be the full dg-subcategory of  perfect complexes on $X$.

Let $K(X)$ be the homotopy category of complexes of $\mathcal{O}_X$-modules on $X$. Then  $K_{\text{perf}}(X)$ is the triangulated subcategories of $K(X)$ which consists of perfect complexes of $\mathcal{O}_X$-module.

Moreover let $D(X)$ be the derived category of complexes of $\mathcal{O}_X$-modules on $X$. Then  $D_{\text{perf}}(X)$ is the triangulated subcategory of $D(X)$ which consists of perfect complexes of $\mathcal{O}_X$-modules.
\end{definition}

We need to also consider the complexes of quasi-coherent sheaves on $X$ and we have the following definition.
\begin{definition}\label{defi: derived categories with quasi-coherent components}
Let $\text{Qcoh}(X)$ be the dg-category of complexes of quasi-coherent sheaves on $X$. It is clear that $\text{Qcoh}(X)$ is a full dg-subcategory of $\text{Sh}(X)$. Let $\text{Qcoh}_{\text{perf}}(X)$ be the full dg-subcategory of $\text{Qcoh}(X)$ which consists of perfect complexes of quasi-coherent sheaves. $\text{Qcoh}_{\text{perf}}(X)$ is also a full dg-subcategory of $\text{Sh}_{\text{perf}}(X)$.

Let $K(\text{Qcoh}(X))$ be the homotopy category of $\text{Qcoh}(X)$ and $D(\text{Qcoh}(X))$ be its derived category. Similarly we  have $K_{\text{perf}}(\text{Qcoh}(X))$ and $D_{\text{perf}}(\text{Qcoh}(X))$.
\end{definition}

\begin{remark}\label{remark: quasi-coherent component and general sheaves of A-modules}
We have the natural inclusion $i: \text{Qcoh}(X)\to \text{Sh}(X)$ which induces a functor
$$
\tilde{i}: D(\text{Qcoh}(X))\to D_{\text{Qcoh}}(X),
$$
where $D_{\text{Qcoh}}(X)$ is the derived category of complexes of $\mathcal{O}_X$-modules with quasi-coherent cohomologies. However for general $(X,\mathcal{O}_X)$ the functor $\tilde{i}$ is not  necessarily  essentially surjective nor fully faithful. As a result we need to distinguish complexes of quasi-coherent modules and complexes of general $\mathcal{O}_X$-modules.
This issue will be discussed further in Appendix \ref{appendix: quasi-coherent, pseudo-coherent complexes and coherent sheaves}.
\end{remark}

 \subsection{Notations of bicomplexes and sign conventions}
In this subsection we introduce some notations which are necessary in the definition of twisted complexes, for reference see \cite{o1981trace} Section 1.

Let $(X,\mathcal{O}_X)$ be a locally ringed   space with paracompact underlying topological space  and $\mathcal{U}=\{U_i\}$ be a locally finite open cover of $X$. Let $U_{i_0\ldots i_n}$ denote the intersection $U_{i_0}\cap\ldots \cap U_{i_n}$.

\begin{remark}
\cite{toledo1978duality}, \cite{o1981trace} and \cite{o1985grothendieck} focus on the special case that $X$ is a complex manifold and $\mathcal{O}_X$ is the sheaf of holomorphic functions on $X$. In this paper we consider more  general  $(X,\mathcal{O}_X)$.
\end{remark}

For each $U_{i_k}$, let $E^{\bullet}_{i_k}$ be a   graded sheaf of $\mathcal{O}_X$-modules on $U_{i_k}$. Let
\begin{equation}\label{equation: bigrade sheaves}
C^{\bullet}(\mathcal{U},E^{\bullet})=\prod_{p,q}C^p(\mathcal{U},E^q)
\end{equation}
be the bigraded complexes of $E^{\bullet}$. More precisely, an element $c^{p,q}$ of $C^p(\mathcal{U},E^q)$ consists of a section $c^{p,q}_{i_0\ldots i_p}$ of $E^{q}_{i_0}$ over each non-empty intersection $U_{i_0\ldots i_p}$. If $U_{i_0\ldots i_p}=\emptyset$, simply let the component on it be zero.

Now if another   graded sheaf $F^{\bullet}_{i_k}$ of $\mathcal{O}_X$-modules is given on each $U_{i_k}$, then we can consider the bigraded complex
\begin{equation}\label{equation: map with bigrade between graded sheaves}
C^{\bullet}(\mathcal{U},\text{Hom}^{\bullet}(E,F))=\prod_{p,q}C^p(\mathcal{U},\text{Hom}^q(E,F)).
\end{equation}
An element $u^{p,q}$ of $C^p(\mathcal{U},\text{Hom}^q(E,F))$ gives a section $u^{p,q}_{i_0\ldots i_p}$ of $\text{Hom}^q_{\mathcal{O}_X-\text{Mod}}(E^{\bullet}_{i_p},F^{\bullet}_{i_0})$, i.e. a degree $q$ map from $E^{\bullet}_{i_p}$ to $F^{\bullet}_{i_0}$ over the non-empty intersection $U_{i_0\ldots i_p}$. Notice that we require $u^{p,q}$ to be a map from the $F^{\bullet}$ on the last subscript of $U_{i_0\ldots i_p}$ to the $E^{\bullet}$ on the first subscript of $U_{i_0\ldots i_p}$. Again, if $U_{i_0\ldots i_p}=\emptyset$,  let the component on it be zero.

\begin{remark}\label{remark: (p,q) Cech and cohomology degree}
In this paper when we talk about degree $(p,q)$, the first index always indicates the \v{C}ech degree while the second index always indicates the graded sheaf degree.
\end{remark}

We need to study the compositions of $C^{\bullet}(\mathcal{U},\text{Hom}^{\bullet}(E,F))$. Let $\{G^{\bullet}_{i_k}\}$ be a third   graded sheaf of $\mathcal{O}_X$-modules, then there is a composition map
$$
C^{\bullet}(\mathcal{U},\text{Hom}^{\bullet}(F,G)) \times C^{\bullet}(\mathcal{U},\text{Hom}^{\bullet}(E,F))\to C^{\bullet}(\mathcal{U},\text{Hom}^{\bullet}(E,G)).
$$
In fact, for $u^{p,q}\in C^p(\mathcal{U},\text{Hom}^{q}(F,G))$ and $v^{r,s} \in C^r(\mathcal{U},\text{Hom}^s(E,F))$, their composition $(u\cdot v)^{p+r,q+s}$ is given by (see \cite{o1981trace} Equation (1.1))
\begin{equation}\label{equation: composition of maps between graded sheaves}
(u\cdot v)^{p+r,q+s}_{i_0\ldots i_{p+r}}=(-1)^{qr}u^{p,q}_{i_0\ldots i_p}v^{r,s}_{i_p\ldots i_{p+r}}
\end{equation}
where the right hand side is the na\"{i}ve composition of sheaf maps.

In particular $C^{\bullet}(\mathcal{U},\text{Hom}^{\bullet}(E,E))$ becomes an associative algebra under this composition (It is easy but tedious to check the associativity). We also notice that $C^{\bullet}(\mathcal{U},E^{\bullet})$ becomes a left module over this algebra. In fact the action
$$
C^{\bullet}(\mathcal{U},\text{Hom}^{\bullet}(E,E))\times C^{\bullet}(\mathcal{U},E^{\bullet})\to  C^{\bullet}(\mathcal{U},E^{\bullet})
$$
is given by $(u^{p,q},c^{r,s})\mapsto (u\cdot c)^{p+r,q+s}$ where the action is given by (see \cite{o1981trace} Equation (1.2))
\begin{equation}\label{equation: action of maps on sheaves}
(u\cdot c)^{p+r,q+s}_{i_0\ldots i_{p+r}}=(-1)^{qr}u^{p,q}_{i_0\ldots i_p}c^{r,s}_{i_p\ldots i_{p+r}}
\end{equation}
where the right hand side is given by evaluation.

There is also  a \v{C}ech-style differential operator $\delta$ on $C^{\bullet}(\mathcal{U},\text{Hom}^{\bullet}(E,F))$ and $C^{\bullet}(\mathcal{U},E^{\bullet})$ of bidegree $(1,0)$ given by the formula
\begin{equation}\label{equation: delta on maps}
(\delta u)^{p+1,q}_{i_0\ldots i_{p+1}}=\sum_{k=1}^p(-1)^k u^{p,q}_{i_0\ldots \widehat{i_k} \ldots i_{p+1}}|_{U_{i_0\ldots i_{p+1}}} \,\text{ for } u^{p,q}\in C^p(\mathcal{U},\text{Hom}^q(E,F))
\end{equation}
and
\begin{equation}\label{equation: delta on sheaves}
(\delta c)^{p+1,q}_{i_0\ldots i_{p+1}}=\sum_{k=1}^{p+1}(-1)^k c^{p,q}_{i_0\ldots \widehat{i_k} \ldots i_{p+1}}|_{U_{i_0\ldots i_{p+1}}} \,\text{ for }c^{p,q}\in C^p(\mathcal{U},E).
\end{equation}

\begin{ctn}
Notice that the map $\delta$ defined above is different from the usual \u{C}ech differential. In Equation \eqref{equation: delta on maps} we do not include the $0$th and the $(p+1)$th indices and in Equation \eqref{equation: delta on sheaves} we do not include the $0$th  index.
\end{ctn}

\begin{proposition}\label{prop: Leibniz for Cech differential}
The  differential satisfies the Leibniz rule. More precisely we have
$$
\delta(u\cdot v)=(\delta u)\cdot v+(-1)^{|u|}u\cdot (\delta v)
$$
and
$$
\delta(u\cdot c)=(\delta u)\cdot c+(-1)^{|u|}u\cdot (\delta c)
$$where $|u|$ is  the total degree of $u$.
\end{proposition}
\begin{proof}
This is a routine check.
\qed \end{proof}

\subsection{The definition of twisted complex}
Now we can define twisted complexes.

\begin{definition}\label{defi: twisted complex}[Twisted complexes]
Let $(X,\mathcal{O}_X)$ be a locally ringed, paracompact space and $\mathcal{U}=\{U_i\}$ be a locally finite open cover of $X$. A \emph{twisted complex} consists of a  graded sheaves $E^{\bullet}_{i}$ of  $\mathcal{O}_X$-modules on each $U_i$ together with a collection of morphisms
$$
a=\sum_{k \geq 0} a^{k,1-k}
$$
where $a^{k,1-k}\in C^k(\mathcal{U},\text{Hom}^{1-k}(E,E)) $ such that they satisfy the Maurer-Cartan equation
$$
\delta a+ a\cdot a=0.
$$
More explicitly, for $k\geq 0$
\begin{equation}\label{equation: MC for twisted complex explicit}
\delta a^{k-1,2-k}+ \sum_{i=0}^k a^{i,1-i}\cdot a^{k-i,1-k+i}=0.
\end{equation}

Moreover we impose the following non-degenerate condition: for each $i$, the chain map
$$
a^{1,0}_{ii}: (E^{\bullet}_i,  a^{0,1}_i)\to (E^{\bullet}_i,  a^{0,1}_i) \text { is chain homotopic to the identity map.}
$$

The twisted complexes on $(X,\mathcal{O}_X, \{U_i\})$ form a dg-category: the objects are the twisted complexes $\mathcal{E}=(E^{\bullet}_i,a)$ and the morphisms from $\mathcal{E}=(E^{\bullet}_i,a)$ to $\mathcal{F}=(F^{\bullet}_i,b)$ are $C^{\bullet}(\mathcal{U},\text{Hom}^{\bullet}(E,F))$. The degree of a morphism is given by the total degree of $C^{\bullet}(\mathcal{U},\text{Hom}^{\bullet}(E,F))$. Moreover, the differential of a morphism $\phi$ is given by
$$
d \phi=\delta \phi+b\cdot \phi-(-1)^{|\phi|}\phi\cdot a.
$$

We denote the dg-category of twisted complexes on $(X,\mathcal{O}_X, \{U_i\})$ by Tw$(X,\mathcal{O}_X, \{U_i\})$. If there is no danger of confusion we can simply denote it by Tw$(X)$.
\end{definition}

Actually the first few terms of the Maurer-Cartan Equation \eqref{equation: MC for twisted complex explicit} can be written as
\begin{equation*}
\begin{split}
a^{0,1}_i\cdot a^{0,1}_i&=0\\
a^{0,1}_i\cdot a^{1,0}_{ij}+a^{1,0}_{ij}\cdot a^{0,1}_j&=0\\
-a^{1,0}_{ik}+a^{1,0}_{ij}\cdot a^{1,0}_{jk}+a^{0,1}_i\cdot a^{2,-1}_{ijk}+a^{2,-1}_{ijk}\cdot a^{0,1}_k&=0\\
\ldots
\end{split}
\end{equation*}
Let us explain the  meaning of these equations. The first equation tells us that for each $i$, $(E^{\bullet}_i,a^{0,1}_i)$ is a chain complex. The second equation, together with the sign convention in Equation \eqref{equation: composition of maps between graded sheaves}, tells us that $a^{1,0}_{ij}$ gives a chain map $(E^{\bullet}_j,a^{0,1}_j)\to (E^{\bullet}_i,a^{0,1}_i)$. The third equation says that we have the cocycle condition
$$
a^{1,0}_{ik}=a^{1,0}_{ij}  a^{1,0}_{jk}
$$
\emph{up to homotopy} with the homotopy operator $a^{2,-1}_{ijk}$.

\begin{ctn}\label{ctn: twisted complex is not a sheaf}
Notice that a  twisted complex itself is not a complex of sheaves on $X$.
\end{ctn}

For our purpose we need the following smaller dg-categories.
\begin{definition}\label{defi: twisted perfect complex}
A \emph{twisted perfect complex} $\mathcal{E}=(E^{\bullet}_i,a)$ is the same as twisted complex except that each $E^{\bullet}_i$ is required to be a strictly perfect complex on $U_i$.

The twisted perfect complexes form a dg-category and we denote it by $\text{Tw}_{\text{perf}}(X,\mathcal{O}_X, \{U_i\})$ or simply $\text{Tw}_{\text{perf}}(X)$. Obviously $\text{Tw}_{\text{perf}}(X)$ is a full dg-subcategory of Tw$(X)$.
\end{definition}

\begin{remark}\label{remark: aii is not id}
The twisted perfect complex in this paper is almost the same as the \emph{twisted cochain} in \cite{o1981trace}. The only difference between our definition and theirs is that we do not require that for any $i$
$$
a^{1,0}_{ii}=id_{E^{\bullet}_i} \text{ on the nose}.
$$
Our definition guarantees that the mapping cone exists in the category Tw$(X)$, see Definition \ref{defi: mapping cone} below.
\end{remark}

\begin{remark}\label{remark: relation to other topics}
We would like to mention some related topics here.
\begin{itemize}
\item Our construction is very similar to the twisted complex in \cite{bondal1990enhanced}. For example both constructions involve the Maurer-Cartan equation. The main difference is that  the differential of the Maurer-Cartan equation in Bondal and Kapranov's twisted complex is  the differential in the dg-category, while our differential is the \v{C}ech differential $\delta$.

\item The construction of twisted complexes is very similar to the \emph{dg-nerve} as in \cite{luriehigher} 1.3.1.6 or Definition 2.3 in \cite{block2014higher}. It is worthwhile to find the deeper relations.

\item We expect the dg-category $\text{Tw}_{\text{perf}}(X)$ gives an explicit realization of the  \emph{homotopy limit} of $L(U_i)$, the dg-categories of locally free finitely generated sheaves on $U_i$. This problem has been solved in the recent preprint \cite{block2015explicit} and  the construction depends heavily on the simplicial resolution of dg-categories in \cite{holstein2014properness}.
\end{itemize}\end{remark}

\begin{definition}\label{defi: differential deltaa}
For a  fixed twisted complex $(E^{\bullet},a)$, we can define an operator $\delta_a$ on $C^{\bullet}(\mathcal{U},E^{\bullet})$ of total degree $1$ by
$$
\delta_a c=\delta c+ a\cdot c.
$$

The Maurer-Cartan equation $\delta a+a\cdot a=0$ implies that $\delta_a^2=0$, i.e. $\delta_a$ is a differential on  $C^{\bullet}(\mathcal{U},E^{\bullet})$.

We have the same construction when we restrict to $\text{Tw}_{\text{perf}}(X)$.
\end{definition}

\begin{remark}
We see that the differential $\delta_a$ in Definition \ref{defi: differential deltaa} is a twist of the differential $\delta$. This justifies the name "twisted complex".
\end{remark}

\subsection{Further study of the non-degeneracy condition of twisted complexes}
 Recall that for each $i$, the $(0,1)$ component $a^{0,1}_{i}: E^n_{i}\to E^{n+1}_{i}$ is a differential of $\mathcal{O}_X$-modules on $U_{i}$, hence we get a complex $( E^n_i, a^{0,1}_i)$ on $U_i$. Remember that the map is the dot multiplication of $a^{0,1}_{i}$ as in Equation \eqref{equation: action of maps on sheaves}.

 Now we consider the map $a^{1,0}_{ii}: E^n_i\to E^n_i$, the Maurer-Cartan equation \eqref{equation: MC for twisted complex explicit} in the $k=1$ case tells us
 $$
 a^{1,0}_{ii}\cdot a^{0,1}_i+ a^{0,1}_i\cdot a^{1,0}_{ii}=0.
 $$
Actually under the sign convention in Equation \eqref{equation: composition of maps between graded sheaves}, the above equation becomes
$$
 a^{1,0}_{ii}  a^{0,1}_i- a^{0,1}_i  a^{1,0}_{ii}=0.
$$
In other words, $ a^{1,0}_{ii}$ gives a chain map $( E^n_i, a^{0,1}_i)\to (E^n_i, a^{0,1}_i)$.

Let us denote the homotopy category of complexes of $\mathcal{O}_X$-modules on $U_i$ by $K(U_i)$. Then we have the following lemma.

\begin{lemma}\label{lemma: aii is an idempotent}
If the $a^{k,1-k}$'s satisfy the Maurer-Cartan equation, then $ a^{1,0}_{ii}: ( E^n_i, a^{0,1}_i)\to (E^n_i, a^{0,1}_i)$ is an idempotent map in the homotopy category $K(U_i)$, i.e. $(a^{1,0}_{ii})^2=a^{1,0}_{ii}$ up to chain homotopy.
\end{lemma}

\begin{proof}
The $k=2$ case of the Maurer-Cartan equation \eqref{equation: MC for twisted complex explicit} gives us
$$
-a^{1,0}_{ii}+a^{1,0}_{ii}\cdot a^{1,0}_{ii}+a^{0,1}_i\cdot a^{2,-1}_{iii}+a^{2,-1}_{iii}\cdot a^{0,1}_i=0.
$$
We take $a^{2,-1}_{iii}$ to be the homotopy operator and this immediately gives what we want.
\qed \end{proof}

For later purpose we need the following lemma.
\begin{lemma}\label{lemma: homotopy to id is the same as homotopy invertible}
If the $a^{k,1-k}$'s satisfy the Maurer-Cartan equation, then $ a^{1,0}_{ii}: ( E^n_i, a^{0,1}_i)\to (E^n_i, a^{0,1}_i)$ is homotopic to the identity map  if and only if it is homotopic invertible.
\end{lemma}
\begin{proof}
By Lemma \ref{lemma: aii is an idempotent}, we know that $(a^{1,0}_{ii})^2=a^{1,0}_{ii}$  up to chain homotopy. Then the result is obvious.
\qed \end{proof}

We will discuss the non-degeneracy condition further in Section \ref{subsection: splitting}.

\subsection{The pre-triangulated structure on Tw$(X)$}\label{subsection: pre-triangulated structure}
The dg-category Tw$(X)$ has a natural shift-by-one functor and a mapping cone construction as follows.

\begin{definition}\label{defi: shift of twisted complex}[Shift]
Let $\mathcal{E}=(E^{\bullet}_i,a)$ be a twisted complex. We define its shift   $\mathcal{E}[1]$ to be $\mathcal{E}[1]=(E[1]^{\bullet}_i,a[1])$ where
$$
E[1]^{\bullet}_i=E^{\bullet+1}_i \text{ and } a[1]^{k,1-k}=(-1)^{k-1}a^{k,1-k}.
$$

Moreover, let $\phi: \mathcal{E}\to \mathcal{F}$ be a morphism. We define its shift $\phi[1]$ as
$$
\phi[1]^{p,q}=(-1)^q\phi^{p,q}.
$$
\end{definition}

\begin{definition}\label{defi: mapping cone}[Mapping cone]
Let $\phi^{\bullet,-\bullet}$ be a closed degree zero map between twisted complexes $\mathcal{E}=(E^{\bullet},a^{\bullet,1-\bullet})$ and $\mathcal{F}=(F^{\bullet},b^{\bullet,1-\bullet})$ , we can define the \emph{mapping cone} $\mathcal{G}=(G,c)$ of $\phi$ as follows (see \cite{o1985grothendieck} Section 1.1):
$$
G^n_i:=E^{n+1}_i\oplus F^n_i
$$
and
\begin{equation}\label{equation: diff in mapping cone}
c^{k,1-k}_{i_0\ldots i_k}=\begin{pmatrix}(-1)^{k-1} a^{k,1-k}_{i_0\ldots i_k}&0\\ (-1)^k\phi^{k,-k}_{i_0\ldots i_k}&b^{k,1-k}_{i_0\ldots i_k}\end{pmatrix}.
\end{equation}
\end{definition}

\begin{remark}\label{remark: mapping cone cannot have id transition functions}
As a special case of Equation \eqref{equation: diff in mapping cone} we get
$$
c^{1,0}_{ii}=\begin{pmatrix}a^{1,0}_{ii}&0\\ -\phi^{1,-1}_{ii}&b^{1,0}_{ii}\end{pmatrix}.
$$
It is clear that $c^{1,0}_{ii}\neq id$ even if both  $a^{1,0}_{ii}$ and $b^{1,0}_{ii}$ equal to $id$ since we cannot assume that $\phi^{1,-1}_{ii}=0$ for any $i$. This is the main technical reason that we drop the requirement $a^{1,0}_{ii}=id$ in the definition of twisted complex, see Remark \ref{remark: aii is not id}  after Definition \ref{defi: twisted complex}.
\end{remark}

Nevertheless, we can prove that the mapping cone satisfies the non-degeneracy condition in Definition \ref{defi: perfect complex}.

\begin{lemma}\label{lemma: nondegenerate of mapping cone}
Let $\phi^{\bullet,-\bullet}$ be a closed degree zero map between twisted complexes $\mathcal{E}=(E^{\bullet},a^{\bullet,1-\bullet})$ and $\mathcal{F}=(F^{\bullet},b^{\bullet,1-\bullet})$. Let $\mathcal{G}=(G,c)$ be the mapping cone of $\phi$. Then
$$
c^{1,0}_{ii}: (G^{\bullet}_i,c^{0,1}_i)\to (G^{\bullet}_i,c^{0,1}_i)
$$
is chain homotopic to id.
\end{lemma}
\begin{proof}
By Lemma \ref{lemma: homotopy to id is the same as homotopy invertible}, we know that $a^{1,0}_{ii}$ and $b^{1,0}_{ii}$ are homotopic invertible, hence
$$
c^{1,0}_{ii}=\begin{pmatrix}a^{1,0}_{ii}&0\\ -\phi^{1,-1}_{ii}&b^{1,0}_{ii}\end{pmatrix}
$$
is also homotopic invertible since it is a lower block triangular matrix.

On the other hand the $c^{k,1-k}$'s satisfy the Maurer-Cartan equation. Again by Lemma  \ref{lemma: homotopy to id is the same as homotopy invertible} we know that $c^{1,0}_{ii}$ is chain homotopic to id.
\qed \end{proof}

\begin{proposition}\label{prop:pre-triangulated}
Tw$(X)$ is a pre-triangulated dg-category and Tw$(X)$ is a pre-triangulated dg-subcategory of Tw$(X)$. Therefore the category HoTw$(X)$ is triangulated.

The same result holds for $\text{Tw}_{\text{perf}}(X)$.
\end{proposition}
\begin{proof}
It is easy to check this result.
\qed \end{proof}

\begin{ctn}\label{ctn: sign different in mapping cone of Toledo Tong}
The degree and sign convention in the definition of mapping cones in this paper are slightly different to those in \cite{o1985grothendieck} Section 1.1.
\end{ctn}

\subsection{Weak equivalences in Tw$(X)$}\label{subsection: weak equivalence}

In this subsection we specify the class of weak equivalences in Tw$(X)$, which is very important in our later constructions.

\begin{definition}\label{defi: quasi-isomorphism in Tw}[Weak equivalence]
Let $\mathcal{E}=(E^{\bullet},a^{\bullet,1-\bullet})$ and $\mathcal{F}=(F^{\bullet},b^{\bullet,1-\bullet})$ be two objects in Tw$(X)$. A  morphism $\phi: \mathcal{E} \to \mathcal{F}$ is called a \emph{weak equivalence} if it satisfies the following two conditions.
\begin{enumerate}
\item $\phi$ is closed and of degree zero;
\item its $(0,0)$ component
$$
\phi^{0,0}_i: (E^{\bullet}_i,a^{0,1}_i)\to (F^{\bullet}_i,b^{0,1}_i)
$$
is a quasi-isomorphism of complexes of $\mathcal{O}_X$-modules on $U_i$ for each $i$.
\end{enumerate}
\end{definition}

\begin{remark}
The definition of weak equivalence between twisted complexes is first introduced in \cite{gillet1986k}.
\end{remark}

If $\mathcal{E}$ and $\mathcal{F}$ are both in the subcategory $\text{Tw}_{\text{perf}}(X)$ we have a further result on weak equivalence between them. For this we need some assumption on the open cover $\{U_i\}$ and some technical lemmas, which we introduce here.

\begin{lemma}\label{lemma: acyclic of complex on good covers}
Let $U$ be a subset of $X$ which satisfies $H^k(U,\mathcal{F})=0$ for any quasi-coherent sheaf $\mathcal{F}$ on $U$ and any $k\geq 1$. Let $E^{\bullet}$ be a bounded above complex of finitely generated locally free sheaves on $U$ and $G^{\bullet}$ be an acyclic complex of quasi-coherent modules on $U$, then the Hom complex
$\text{Hom}^{\bullet}(E,G)$ is acyclic.
\end{lemma}
\begin{proof}
We have a filtration on $\text{Hom}^{\bullet}(E,G)$ given by the $E^{\bullet}$ degree. More explicitly let
$$
F^k \text{Hom}^{\bullet}(E,G)=\{\phi\in\text{Hom}^{\bullet}(E,G)|\phi(e)=0 \text{ if deg}(e)>-k  \}.
$$

By a simple spectral sequence argument, it is sufficient to prove that
$$(F^k\text{Hom}^{\bullet}(E,G)/F^{k+1}\text{Hom}^{\bullet}(E,G),d_{\text{Hom}})$$ is acyclic for each $k$. We notice that
$$
(F^k\text{Hom}^{\bullet}(E,G)/F^{k+1}\text{Hom}^{\bullet}(E,G),d_{\text{Hom}})\cong (\text{Hom}(E^k,G^{\bullet}), d_G).
$$

We know that $(G^{\bullet},d_G)$ is acyclic. On the other hand $E^k$ is locally free finitely generated hence the assumption in the lemma guarantees that $\text{Hom}(E^k,-)$ is an exact functor, hence we get the acyclicity of $\text{Hom}^{\bullet}(E,G)$.
\qed \end{proof}

\begin{lemma}\label{lemma: homotopy inverse on good covers}
Let $U$ be a subset of $X$ which satisfies $H^k(U,\mathcal{F})=0$ for any quasi-coherent sheaf $\mathcal{F}$ on $U$ and any $k\geq 1$.  Suppose we have chain maps $r: E^{\bullet}\to F^{\bullet}$ and $s: G^{\bullet}\to F^{\bullet}$ between complexes of sheaves on $U$, where $E^{\bullet}$ is a bounded above complex of finitely generated locally free sheaves, and $F^{\bullet}$ and $G^{\bullet}$ are quasi-coherent. Moreover $s$ is a quasi-isomorphism.  Then $r$ factors through $s$ up to homotopy, i.e. there exists a chain map $r^{\prime}: E^{\bullet}\to G^{\bullet}$ such that $s\circ r^{\prime}$ is homotopic to $r$.
\end{lemma}
\begin{proof}
We can take the mapping cone of $s$, which is acyclic, then the result is a simple corollary of Lemma \ref{lemma: acyclic of complex on good covers}.
\qed \end{proof}

With these lemmas we have the following result for twisted perfect complexes.

\begin{proposition}\label{prop: homotopy invertible morphisms}
Let the cover $\{U_i\}$  satisfies   $H^k(U_i,\mathcal{F})=0$ for any $i$, any quasi-coherent sheaf $\mathcal{F}$ on $U_i$ and any $k\geq 1$. If $\mathcal{E}$ and $\mathcal{F}$ are both in the subcategory $\text{Tw}_{\text{perf}}(X)$, then a closed degree zero morphism $\phi$ between twisted complexes $\mathcal{E}$ and $\mathcal{F}$ is  a weak equivalence  if and only if $\phi$ is invertible in the homotopy category $\text{HoTw}_{\text{perf}}(X)$.
\end{proposition}
\begin{proof}
It is obvious that homotopy invertibility implies weak equivalence.

For the other direction, we know $\phi$ is a weak equivalence, hence $\phi^{0,0}_i: E^{\bullet}_i\to F^{\bullet}_i$ is a quasi-isomorphism for each $i$. Since $F^{\bullet}_i$ is a bounded complex of finitely generated locally free sheaves, we apply Lemma \ref{lemma: homotopy inverse on good covers} and get
$$
\psi^{0,0}_i: F^{\bullet}_i\to E^{\bullet}_i
$$
such that $\phi^{0,0}_i\circ \psi^{0,0}_i$ is homotopic to $id_{F^{\bullet}_i}$. It is clear that $\psi^{0,0}_i$ is also a quasi-isomorphism and gives the two-side homotopy inverse of $\phi^{0,0}_i$.

The remaining task is to extend the $\psi^{0,0}_i$'s to a degree zero cocycle $\psi^{\bullet,-\bullet}$ in Tw$(X)$ and to show that it gives the homotopy inverse of $\phi^{\bullet,-\bullet}$. This is a simple spectral sequence argument which is the same as the proof of Proposition 2.9 in \cite{block2010duality}.
\qed \end{proof}

\begin{remark}
The result of Proposition \ref{prop: homotopy invertible morphisms} is no longer true if one of  $\mathcal{E}$ and $\mathcal{F}$ is not a twisted perfect complex.
\end{remark}

We also have the following result.

\begin{proposition}\label{prop: homotopy inverse for twisted complexes}
Let $\{U_i\}$ be an open cover of $X$ such that for any finite intersection $U_I$ we have $H^k(U_I, \mathcal{F})=0$ for any quasi-coherent sheaf $\mathcal{F}$ on $U_I$ and any $k\geq 1$. Let $\mathcal{E}$ be a twisted perfect complex and $\mathcal{F}$, $\mathcal{G}$ be twisted complexes consisting of quasi-coherent sheaves on each $U_i$. Let $\varphi: \mathcal{G}\to \mathcal{F}$ be a weak equivalence. Then any closed morphism $\phi: \mathcal{E}\to \mathcal{F}$ factors through $\varphi$ up to homotopy, i.e.  there exists a chain map $\theta: \mathcal{E}\to \mathcal{G}$ such that $\varphi \cdot \theta$ is homotopic to $\phi$.
\end{proposition}
\begin{proof}
The proof is inspired by that of Proposition 1.2.3 in \cite{o1985grothendieck}, see also Proposition \ref{prop: existence of resolution} below.

First we fix the notation. Let $\mathcal{E}=(E^{\bullet}_i,a)$, $\mathcal{F}=(F^{\bullet}_i,b)$, and $\mathcal{G}=(G^{\bullet}_i,c)$. Let $l$ be the degree of $\phi: \mathcal{E}\to \mathcal{F}$.

Since  $\varphi: \mathcal{G}\to \mathcal{F}$ is a weak equivalence, we know that on each $U_i$, $\varphi^{0,0}: G^{\bullet}_i\to F^{\bullet}_i$ is a quasi-isomorphism of complexes of quasi-coherent sheaves. By Lemma \ref{lemma: homotopy inverse on good covers} we know that $\phi^{0,l}_i: E^{\bullet}_i\to F^{\bullet+l}_i$  factors through $\varphi^{0,0}_i$ up to homotopy, i.e. there exist $\theta^{0,l}_i:E^{\bullet}_i\to G^{\bullet+l}_i$ and $\mu^{0,l-1}_i: E^{\bullet}_i\to F^{\bullet+l-1}_i$ such that
$$
c^{0,1}_i\theta^{0,l}_i-\theta^{0,l}_ia^{0,1}_i=0
$$
and
$$
\varphi^{0,0}_i\theta^{0,l}_i-\phi^{0,l}_i=b^{0,1}_i\mu^{0,l-1}_i-\mu^{0,l-1}_ia^{0,1}_i.
$$

Now we need to do the following two constructions:
\begin{enumerate}
\item Extend $\theta^{0,l}_i$ to a closed map $\theta: \mathcal{E}\to \mathcal{G}$ between twisted complexes.
\item Extend $\mu^{0,l-1}_i$ to a homotopy between $\varphi\cdot \theta$ and $\phi$.
\end{enumerate}

In more details, on each $U_{i_0\ldots i_k}$ we need to find $\theta^{k,l-k}_{i_0\ldots i_k}: E^{\bullet}_{i_k}\to G^{\bullet+l-k}_{i_0}$ and $\mu^{k,l-1-k}_{i_0\ldots i_k}:  E^{\bullet}_{i_k}\to F^{\bullet+l-1-k}_{i_0}$ such that
\begin{equation}\label{equation: lift is a closed map}
\sum_{j=1}^{k-1}(-1)^j\theta^{k-1,l+1-k}_{i_0\ldots \widehat{i_j}\ldots i_k}+\sum_{j=0}^k c^{j,1-j}_{i_0\ldots i_j}\cdot \theta^{k-j,l+j-k}_{i_j\ldots i_k}-(-1)^l\sum_{j=0}^k \theta^{j,l-j}_{i_0\ldots i_j}\cdot a^{k-j,1+j-k}_{i_j\ldots i_k}=0
\end{equation}
and
\begin{equation}\label{equation: lift is chain homotopic}
\begin{split}
\sum_{j=0}^k&\varphi^{j,-j}_{i_0\ldots i_j}\cdot  \theta^{k-j,l+j-k}_{i_j\ldots i_k}-\phi^{k,l-k}_{i_0\ldots i_k}=\\
&\sum_{j=1}^{k-1}(-1)^j\mu^{k-1,l-k}_{i_0\ldots \widehat{i_j}\ldots i_k}+\sum_{j=0}^k b^{j,1-j}_{i_0\ldots i_j}\cdot \mu^{k-j,l-1+j-k}_{i_j\ldots i_k}+(-1)^l\sum_{j=0}^k \mu^{j,l-1-j}_{i_0\ldots i_j}\cdot a^{k-j,1+j-k}_{i_j\ldots i_k}
\end{split}
\end{equation}

We use induction to find the $\theta$'s and $\mu$'s. First remember that the $\theta^{0,l}$'s and the $\mu^{0,l-1}$'s have already been achieved. Now assume that for any multi-index $I$ with cardinality $|I|<k+1$ we have found the $\theta$ and $\mu$ on $U_I$ and they satisfy Equation \eqref{equation: lift is a closed map} and Equation \eqref{equation: lift is chain homotopic} on $U_I$.

Then we need to find $\theta^{k,l-k}_{i_0\ldots i_k}$ and $\mu^{k,l-1-k}_{i_0\ldots i_k}$. To do this we consider the mapping cone of $\varphi$ and denote it by $\mathcal{S}=(S^n_i,s)$. By definition we know that $S^n_i=G^{n+1}_i\oplus F^n_i$ and the s's are given by
$$
s^{k,1-k}_{i_0\ldots i_k}=\begin{pmatrix}(-1)^{k-1} c^{k,1-k}_{i_0\ldots i_k}&0\\ (-1)^k\varphi^{k,-k}_{i_0\ldots i_k}&b^{k,1-k}_{i_0\ldots i_k}\end{pmatrix}.
$$
In particular on each $U_i$ we have
$$
s^{0,1}_i= \begin{pmatrix}- c^{0,1}_i&0\\  \varphi^{0,0}_i&b^{0,1}_i\end{pmatrix}.
$$
Since $\varphi: \mathcal{G}\to \mathcal{F}$ is a weak equivalence, we know that for each $U_i$, $(S^{\bullet}_i,s^{0,1}_i)$ is an acyclic  complex of quasi-coherent sheaves. By Lemma \ref{lemma: acyclic of complex on good covers} $\text{Hom}^{\bullet}(E^{\bullet}_i,S^{\bullet}_i)$ is also acyclic. Moreover, $\text{Hom}^{\bullet}(E^{\bullet}_{i_k},S^{\bullet}_{i_0})$ is acyclic on $U_{i_0\ldots i_k}$.

Then we rearrange Equation \eqref{equation: lift is a closed map} and Equation \eqref{equation: lift is chain homotopic} and get the following equations.
\begin{equation}\label{equation: lift is a closed map induction}
\begin{split}
\sum_{j=1}^{k-1}(-1)^j\theta^{k-1,l+1-k}_{i_0\ldots \widehat{i_j}\ldots i_k}+&\sum_{j=1}^k c^{j,1-j}_{i_0\ldots i_j}\cdot \theta^{k-j,l+j-k}_{i_j\ldots i_k}-(-1)^l\sum_{j=0}^{k-1} \theta^{j,l-j}_{i_0\ldots i_j}\cdot a^{k-j,1+j-k}_{i_j\ldots i_k}\\=
&(-1)^l \theta^{k,l-k}_{i_0\ldots i_k}\cdot a^{0,1}_{i_k}-c^{0,1}_{i_0}\cdot \theta^{k,l-k}_{i_0\ldots i_k}
\end{split}
\end{equation}
and
\begin{equation}\label{equation: lift is chain homotopic induction}
\begin{split}
&-\sum_{j=1}^k\varphi^{j,-j}_{i_0\ldots i_j}\cdot  \theta^{k-j,l+j-k}_{i_j\ldots i_k}+\phi^{k,l-k}_{i_0\ldots i_k}+\sum_{j=1}^{k-1}(-1)^j\mu^{k-1,l-k}_{i_0\ldots \widehat{i_j}\ldots i_k}\\
&+\sum_{j=1}^k b^{j,1-j}_{i_0\ldots i_j}\cdot \mu^{k-j,l-1+j-k}_{i_j\ldots i_k}+(-1)^l\sum_{j=0}^{k-1}\mu^{j,l-1-j}_{i_0\ldots i_j}\cdot a^{k-j,1+j-k}_{i_j\ldots i_k}\\
&= \varphi^{0,0}_{i_0}\cdot  \theta^{k,l-k}_{i_0\ldots i_k}-b^{0,1}_{i_0}\cdot \mu^{k,l-1-k}_{i_0\ldots i_k}-(-1)^l\mu^{k,l-1-k}_{i_0\ldots i_k}\cdot a^{0,1}_{i_k}.
\end{split}
\end{equation}

We denote the left hand side of Equation \eqref{equation: lift is a closed map induction} and Equation \eqref{equation: lift is chain homotopic induction} by $\Theta$ and $\Xi$ respectively. Notice that $\Theta$ and $\Xi$ do not involve $\theta^{k,l-k}_{i_0\ldots i_k}$ and $\mu^{k,l-1-k}_{i_0\ldots i_k}$. Moreover by induction assumption we can check that
$$
-c^{0,1}_{i_0}\cdot \Theta -(-1)^l\Theta\cdot a^{0,1}_{i_k}=0
$$
and
$$
\varphi^{0,0}_{i_0}\cdot\Theta+b^{0,1}_{i_0}\cdot\Xi-(-1)^l\Xi \cdot a^{0,1}_{i_k}=0.
$$
In other words $(\Theta,\Xi): E^{\bullet}_{i_k}\to S^{\bullet+l-k}_{i_0}$ is closed. Since $\text{Hom}^{\bullet}(E^{\bullet}_{i_k},S^{\bullet}_{i_0})$ is acyclic, we can find
$$
(\theta^{k,l-k}_{i_0\ldots i_k},\mu^{k,l-1-k}_{i_0\ldots i_k}): E^{\bullet}_{i_k}\to S^{\bullet+l-1-k}_{i_0}
$$
such that
$$
-c^{0,1}_{i_0}\cdot \theta^{k,l-k}_{i_0\ldots i_k}+(-1)^l\theta^{k,l-k}_{i_0\ldots i_k}\cdot a^{0,1}_{i_k}=\Theta
$$
and
$$
\varphi^{0,0}_{i_0}\cdot  \theta^{k,l-k}_{i_0\ldots i_k}-b^{0,1}_{i_0}\cdot \mu^{k,l-1-k}_{i_0\ldots i_k}-(-1)^l\mu^{k,l-1-k}_{i_0\ldots i_k}\cdot a^{0,1}_{i_k}=\Xi.
$$
In other words, Equation \eqref{equation: lift is a closed map induction} and Equation \eqref{equation: lift is chain homotopic induction} hold. We have finished the proof.
\qed \end{proof}

\begin{remark}
Proposition  \ref{prop: homotopy invertible morphisms} and \ref{prop: homotopy inverse for twisted complexes} are not explicitly given in \cite{toledo1978duality}, \cite{o1981trace}, \cite{o1985grothendieck}.
\end{remark}

\section{TWISTED COMPLEXES AND THE DG-ENHANCEMENT OF $D_{\text{perf}}(X)$}\label{section: twisted complex and dg-enhancement}

\subsection{The sheafification functor $\mathcal{S}$}\label{subsection: sheafification functor}
In this section we come to our main topic in this paper. First we fix a locally finite open cover   $\mathcal{U}=\{U_i\}$ of $X$. As we noticed in Caution \ref{ctn: twisted complex is not a sheaf}, a  twisted complex $\mathcal{E}=(E^{\bullet}_i,a)$ is not a complex of sheaves. Nevertheless in this subsection we associate a complex of sheaves to each twisted complex on $X$.

First we introduce a variation of the notations in Equation \eqref{equation: bigrade sheaves} and \eqref{equation: map with bigrade between graded sheaves}. Let $E^{\bullet}_{i_k}=\{E^{r}_{i_k}\}_{r\in \mathbb{Z}}$ be a graded sheaf of $\mathcal{O}_X$-modules on $U_{i_k}$ as before.  For $V$ an open subset of $X$, let
$$
C^{\bullet}(\mathcal{U},E^{\bullet};V)=\prod_{p,q}C^p(\mathcal{U},E^q;V)
$$
be the bigraded complex on $V$. More precisely, an element $c^{p,q}$ of $C^p(\mathcal{U},E^q;V)$ consists of a section $c^{p,q}_{i_0\ldots i_p}$ of $E^{q}_{i_0}$ over each non-empty intersection $U_{i_0\ldots i_n}\cap V$. If $U_{i_0\ldots i_n}\cap V=\emptyset$, let the component on $U_{i_0\ldots i_n}\cap V$ simply be zero.

Similarly if another  graded sheaf $F^{\bullet}_{i_k}$ of $\mathcal{O}_X$-modules is given on each $U_{i_k}$, and $V$ is an open subset of $X$, we can consider the bigraded complex
$$
C^{\bullet}(\mathcal{U},\text{Hom}^{\bullet}(E,F);V)=\prod_{p,q}C^p(\mathcal{U},\text{Hom}^q(E,F);V).
$$
An element $u^{p,q}$ of $C^p(\mathcal{U},\text{Hom}^q(E,F);V)$ gives a section $u^{p,q}_{i_0\ldots i_p}$ of $\text{Hom}^q_{\mathcal{O}_X-\text{Mod}}(E^{\bullet}_{i_p},F^{\bullet}_{i_0})$ over each non-empty intersection $U_{i_0\ldots i_n}\cap V$. If $U_{i_0\ldots i_n}\cap V=\emptyset$, let the component on $U_{i_0\ldots i_n}\cap V$ simply  be zero.

Moreover, let $\mathcal{E}=(E^{\bullet}_i,a)$ be a  twisted complex, recall that in Definition \ref{defi: differential deltaa} we defined a differential
$$
\delta_a=\delta+a
$$
on $C^{\bullet}(\mathcal{U},E^{\bullet})$. Now let $V$ be an open subset of $X$, we can restrict $\delta_a$ to $V$ to get a differential on $C^{\bullet}(\mathcal{U},E^{\bullet};V)$.

With all these notations, we can introduce the following definition.

\begin{definition}\label{defi: sheaf associated to twisted complex}
For a twisted complex $\mathcal{E}=(E^{\bullet}_i,a)$, we define the associated complex of sheaves $\mathcal{S}(\mathcal{E})$ as follows: for each $n$, the degree $n$ part $\mathcal{S}^n(\mathcal{E})$ is a sheaf on $X$ such that for any open subset $V$ of $X$
$$
\mathcal{S}^n(\mathcal{E})(V)=\prod_{p+q=n}C^p(\mathcal{U},E^q;V).
$$

The differential on $\mathcal{S}^{\bullet}(\mathcal{E})$ is defined to be the sheafification of $\delta_a=\delta+a$.  More precisely, for each open subset $V$ of $X$, the differential
$$
\mathcal{S}^n(\mathcal{E})(V)\to \mathcal{S}^{n+1}(\mathcal{E})(V)
$$
is given by $\delta+a$ restricted to $V$. We still denote it by $\delta_a$ since there is no danger of confusion.

It is obvious that $\mathcal{S}^n(\mathcal{E})$ is a sheaf of $\mathcal{O}_X$-module for each $n$ and $\delta_a: \mathcal{S}^n(\mathcal{E}) \to \mathcal{S}^{n+1}(\mathcal{E})$ is a map of $\mathcal{O}_X$-modules.
\end{definition}

Now we turn to the morphisms. Let $\phi: \mathcal{E}\to \mathcal{F}$ be a degree $n$ morphism in Tw$(X)$. We can define the associated sheaf morphism
$$
\mathcal{S}(\phi): \mathcal{S}^{\bullet}(\mathcal{E})\to \mathcal{S}^{\bullet+n}(\mathcal{F})
$$ in the same spirit as Definition \ref{defi: sheaf associated to twisted complex}, i.e. by restricting to each of the $C^p(\mathcal{U},E^q;V)$'s.

In fact we can view $\mathcal{S}^{\bullet}(\mathcal{E})$ in another way. For this we recall some definitions in sheaf theory. Let $\mathcal{F}$ be any sheaf of $\mathcal{O}_X$-modules on $X$ and $U$ be an open subset of $X$ with $j: U\to X$ be the inclusion map. We denote the restriction sheaf of $\mathcal{F}$ on $U$ by $\mathcal{F}|_{U}$. The pushforward of $\mathcal{F}|_{U}$ is denoted by $j_*(\mathcal{F}|_{U})$ and it will be a sheaf of $\mathcal{O}_X$-modules on $X$ again and we also denote it by $\mathcal{F}|_{U}$ if there is no confusion.

\begin{remark}\label{remark: support of restriction}
We do not use the fancy pushforward $j_!$ in this paper.
\end{remark}

Then we have
\begin{equation}\label{equation: sheafification functor: alternative}
\mathcal{S}^n(\mathcal{E})=\prod_{p+q=n}E^q_{i_0}|_{U_{i_0\ldots i_p}}
\end{equation}
as a sheaf and the differential $\delta_a=\delta+a$ and the morphism $\mathcal{S}(\phi)$ are defined likewise by restriction.


 In conclusion we have the following definition.

\begin{definition}\label{defi: sheafification functor}[The sheafification functor]
The above construction defines  a dg-fuctor
$$
\mathcal{S}: \text{Tw}(X)\to \text{Sh}(X)
$$
and we call it the \emph{sheafification functor}.
\end{definition}

\begin{remark}\label{remark: sheafification of twisted perfect complex is quasi-coherent}
If the complexes $E^{\bullet}_i$ are bounded and the cover $\{U_i\}$ is locally finite, it is easy to see that the product in $\mathcal{S}^n(\mathcal{E})=\prod_{p+q=n}E^q_{i_0}|_{U_{i_0\ldots i_p}}$ is locally finite, hence the image of a twisted perfect  complex under $\mathcal{S}$ actually consists of quasi-coherent sheaves. In other words, the sheafification functor restricts to $\text{Tw}_{\text{perf}}(X)$ and gives
$$
\mathcal{S}: \text{Tw}_{\text{perf}}(X)\to \text{Qcoh}(X).
$$

Further study of the sheafification of twisted perfect complexes will be given in the next subsection.
\end{remark}

\subsection{The sheafification of twisted perfect complexes}\label{subsection: perfectness}
Let $\mathcal{E}$ be a twisted perfect  complex, we want to show that the associated complex of sheaves $(\mathcal{S}^{\bullet}(\mathcal{E}),\delta_a)$ is  perfect. In fact in this subsection we will get a more general result. The next proposition, which is important in our work, says that locally $(\mathcal{S}^{\bullet}(\mathcal{E}),\delta_a)$ contains the same information as $(E^{\bullet}_j,a^{0,1}_j)$ for each $j$.

\begin{proposition}\label{prop: associated sheaf locall isom to original twisted complex}[The local property of $\mathcal{S}$]
Let $\mathcal{E}=(E^{\bullet}_i,a)$ be a twisted complex and $(\mathcal{S}^{\bullet}(\mathcal{E}),\delta_a)$ be the associated complex of sheaves. Then for each $U_j$ the complex of sheaves $(\mathcal{S}^{\bullet}(\mathcal{E}),\delta_a)|_{U_j}$ is chain homotopy equivalent to $(E^{\bullet}_{j},a^{0,1}_{j})$, i.e. we have two morphisms
$$
f: (\mathcal{S}^{\bullet}(\mathcal{E}),\delta_a)|_{U_j} \to (E^{\bullet}_{j},a^{0,1}_{j})
$$
and
$$
g: (E^{\bullet}_{j},a^{0,1}_{j}) \to (\mathcal{S}^{\bullet}(\mathcal{E}),\delta_a)|_{U_j}
$$
such that
\begin{equation}\label{equation: splitting for associated sheaves}
f\circ g=id_{E^{\bullet}_j} \text{ and } g\circ f=id_{\mathcal{S}^{\bullet}(\mathcal{E})|_{U_j}} \text{ up to chain homotopy.}
\end{equation}
\end{proposition}
\begin{proof}
The proof is long and involves several technical lemmas.

First we can construct the chain map
$$
f: (\mathcal{S}^{\bullet}(\mathcal{E})(V),\delta_a) \to (E^{\bullet}_{j}(V),a^{0,1}_{j})
$$
for $V\subset U_j$ by projecting to the $(0,n)$ component. In more details, we know that
$$
\mathcal{S}^n(\mathcal{E})(V)=\prod_{p+q=n}C^p(\mathcal{U},E^q;V).
$$
The $(0,n)$ component $C^0(\mathcal{U},E^n;V)$ has a further decomposition
$$
C^0(\mathcal{U},E^n;V)=\prod_{i_0}E^n_{i_0}(V\cap U_{i_0}).
$$
We also notice that $j$ appears in one of the $i_0$'s. Then $f: (\mathcal{S}^{\bullet}(\mathcal{E})(V),\delta_a) \to (E^{\bullet}_{j}(V),a^{0,1}_{j})$ is given by first projecting to the $(0,n)$ component and then projecting to the $j$ component. It is easy to see that $f$ is a chain map.

The construction of the  map in the opposite direction
$$
g: (E^{\bullet}_{j}(V),a^{0,1}_{j}) \to (\mathcal{S}^{\bullet}(\mathcal{E})(V),\delta_a)
$$
is more complicated.  We first introduce the following auxiliary morphism
$$
\epsilon^p_{i_0\ldots i_p}: E^{\bullet}_{i_0}(U_{i_0\ldots i_pj}\cap V)\to E^{\bullet}_{i_0}(U_{i_0\ldots i_p}\cap V)
$$
as
$$
\epsilon^p_{i_0\ldots i_p}=(-1)^p~id.
$$
Sometimes we simply denote it by $\epsilon^{p}$. Since $V\subseteq U_j$, we have $U_{i_0\ldots i_p}\cap V\subseteq U_{i_0\ldots i_pj}$ hence the above formula makes sense.

Notice that the identity map $E^{n-p}_{i_0}(U_{i_0\ldots i_pj}\cap V)\to E^{n-p}_{i_0}(U_{i_0\ldots i_p}\cap V)$ shifts the \v{C}ech degree by $-1$ and hence we introduce the factor $(-1)^p$ to compensate it.

We have the following property of the maps $\epsilon^{\bullet}$'s.
\begin{lemma}\label{lemma: epsilon anticommute with a}
The $\epsilon^{\bullet}$'s   anti-commute with $a$ and $\delta$. More precisely, for a multi-index $i_0,\ldots ,i_{p+q}$, we have
\begin{equation}\label{equation: epsilon anticommute with a}
a^{p,1-p}_{i_0\ldots i_p}\epsilon^q_{i_p\ldots i_{p+q}}=-\epsilon^{p+q}_{i_0\ldots i_{p+q}}a^{p,1-p}_{i_0\ldots i_p}
\end{equation}
where both sides are considered as maps
$$
E^{\bullet}_{i_p}(U_{i_p\ldots i_{p+q}j}\cap V)\to E^{\bullet+1-p}_{i_0}(U_{i_0\ldots i_{p+q}}\cap V).
$$

As for $\delta$, we introduce a map $\tilde{\delta}$ on $U_{i_0\ldots i_pj}\cap V$ as
$$
(\tilde{\delta}c)_{i_0\ldots i_pj}=\sum_{k=1}^p(-1)^k c_{i_0\ldots \widehat{i_k}\ldots i_pj}.
$$
Then we have
\begin{equation}\label{equation: epsilon anticommute with delta}
\delta\epsilon^p=-\epsilon^{p+1}\tilde{\delta}.
\end{equation}
\end{lemma}
\begin{proof}[Proof of Lemma \ref{lemma: epsilon anticommute with a}]
First we prove that Equation \eqref{equation: epsilon anticommute with a} holds. Let $c\in E^{\bullet}_{i_p}(U_{i_p\ldots i_{p+q}j}\cap V)$ be with \v{C}ech degree $q+1$. By definition
$$
\epsilon^q_{i_p\ldots i_{p+q}}c=(-1)^qc\in E^{\bullet}_{i_p}(U_{i_p\ldots i_{p+q}}\cap V)
$$
has \v{C}ech degree $q$. Then according to the sign convention in Equation \eqref{equation: action of maps on sheaves} we have
$$
a^{p,1-p}_{i_0\ldots i_p}\epsilon^q_{i_p\ldots i_{p+q}}c=(-1)^qa^{p,1-p}_{i_0\ldots i_p}\cdot c=(-1)^q(-1)^{(1-p)q}a^{p,1-p}_{i_0\ldots i_p}  c=(-1)^{pq} a^{p,1-p}_{i_0\ldots i_p} ~ c.
$$

On the other hand we have
$$
a^{p,1-p}_{i_0\ldots i_p}\cdot c=(-1)^{(1-p)(1+q)}a^{p,1-p}_{i_0\ldots i_p} ~c
$$
hence
$$
\epsilon^{p+q}_{i_0\ldots i_{p+q}}a^{p,1-p}_{i_0\ldots i_p}\cdot c=(-1)^{p+q}(-1)^{(1-p)(1+q)}a^{p,1-p}_{i_0\ldots i_p} ~c=(-1)^{1+pq}a^{p,1-p}_{i_0\ldots i_p} ~c.
$$
Comparing the two sides we get
$$
a^{p,1-p}_{i_0\ldots i_p}\epsilon^q_{i_p\ldots i_{p+q}}=-\epsilon^{p+q}_{i_0\ldots i_{p+q}}a^{p,1-p}_{i_0\ldots i_p}.
$$

Equation \eqref{equation: epsilon anticommute with delta} follows similarly and we leave it as an exercise.
\qed \end{proof}

We move on to the definition of $g$. Recall that
$$
\mathcal{S}^n(\mathcal{E})(V)=\prod_{p+q=n}C^p(\mathcal{U},E^q;V)=\prod_{p\geq 0}\prod_{i_0\ldots i_p}E^{n-p}_{i_0}(U_{i_0\ldots i_p}\cap V)
$$
and it is sufficient to define the projection of $g$ to each component. With the  help of the map $\epsilon^p$   we define that projection to be
$$
\epsilon^p\circ a^{p+1,-p}_{i_0\ldots i_pj}:E^n_j(V)\to  E^{n-p}_{i_0}(U_{i_0\ldots i_p}\cap V), ~p\geq 0.
$$

\begin{lemma}\label{lemma: g is a chaim map}
The map $g: (E^{\bullet}_{j}(V),a^{0,1}_{j}) \to (\mathcal{S}^{\bullet}(\mathcal{E})(V),\delta_a)$ defined above is a chain map.
\end{lemma}
\begin{proof}[Proof of Lemma \ref{lemma: g is a chaim map}]
It is a consequence of the Maurer-Cartan equation
$$
\delta a^{k-1,2-k}+ \sum_{i=0}^k a^{i,1-i}\cdot a^{k-i,1-k+i}=0
$$
together with the anti-commute properties in Lemma \ref{lemma: epsilon anticommute with a}.
\qed \end{proof}

Now we need to prove that $f$ and $g$ satisfy the relations in Equation \eqref{equation: splitting for associated sheaves}. First it is obvious that
$$
f\circ g=a^{1,0}_{jj}: (E^{\bullet}_{j}(V),a^{0,1}_{j}) \to(E^{\bullet}_{j}(V),a^{0,1}_{j}).
$$
By definition $a^{1,0}_{jj}=id_{E^{\bullet}_j}$ up to homotopy hence we get $f\circ g=id_{E^{\bullet}_j}$ up to homotopy.

The other half is more complicated. We need to build a map
$$
h: \mathcal{S}^{\bullet}(\mathcal{E})(V)\to \mathcal{S}^{\bullet-1}(\mathcal{E})(V)
$$
such that
$$
g\circ f-id=\delta_a h+h\delta_a.
$$

In fact we define $h$ as
$$
(h c)_{i_0\ldots i_k}:=(-1)^k c_{i_0\ldots i_k j}.
$$
Clearly $h$ is a sheaf map with degree $-1$. Moreover we have
\begin{equation*}
\begin{split}
&(\delta_a h c)_{i_0\ldots i_k}\\
=&(\delta (hc))_{i_0\ldots i_k}+(a\cdot (hc))_{i_0\ldots i_k}\\
=& \sum_{l=1}^k(-1)^l(hc)_{i_0\ldots \widehat{i_l}\ldots i_k}+\sum_{l=0}^k a^{l,1-l}_{i_0\ldots i_l}\cdot(hc)_{i_l\ldots i_k}\\
=& \sum_{l=1}^k(-1)^l(-1)^{k-1}c_{i_0\ldots \widehat{i_l}\ldots i_kj}+\sum_{l=0}^k a^{l,1-l}_{i_0\ldots i_l}\cdot(hc)_{i_l\ldots i_k} .
\end{split}
\end{equation*}
For the second term $a^{l,1-l}_{i_0\ldots i_l}\cdot(hc)_{i_l\ldots i_k}$ we need to be more careful. We know that $(hc)_{i_l\ldots i_k}$ has \v{C}ech degree $k-l$ hence
\begin{equation*}
\begin{split}
&a^{l,1-l}_{i_0\ldots i_l}\cdot(hc)_{i_l\ldots i_k}\\
=&(-1)^{(1-l)(k-l)}a^{l,1-l}_{i_0\ldots i_l}\circ(hc)_{i_l\ldots i_k}\\
=&(-1)^{(1-l)(k-l)}(-1)^{k-l}a^{l,1-l}_{i_0\ldots i_l}\circ c_{i_l\ldots i_kj}\\
=& (-1)^{lk-l}a^{l,1-l}_{i_0\ldots i_l}\circ c_{i_l\ldots i_kj}.
\end{split}
\end{equation*}
In conclusion we have
\begin{equation}\label{equation: delta_a h}
(\delta_a h c)_{i_0\ldots i_k}=\sum_{l=1}^k(-1)^{k+l-1}c_{i_0\ldots \widehat{i_l}\ldots i_kj}+\sum_{l=0}^k (-1)^{lk-l}a^{l,1-l}_{i_0\ldots i_l}\circ c_{i_l\ldots i_kj}.
\end{equation}

On the other hand we have
\begin{equation*}
\begin{split}
(h& \delta_a  c)_{i_0\ldots i_k}=(-1)^k(\delta_a  c)_{i_0\ldots i_kj}\\
=&(-1)^k[(\delta c)+(a\cdot c)]_{i_0\ldots i_kj}\\
=&(-1)^k[\sum_{l=1}^k(-1)^lc_{i_0\ldots \widehat{i_l}\ldots i_kj}+(-1)^{k+1}c_{i_0\ldots i_k}+\sum_{l=0}^ka^{l,1-l}_{i_0\ldots i_l}\cdot c_{i_l\ldots i_kj}+a^{k+1,-k}_{i_0\ldots i_kj}\cdot c_j]\\
=&(-1)^k[\sum_{l=1}^k(-1)^lc_{i_0\ldots \widehat{i_l}\ldots i_kj}+(-1)^{k+1}c_{i_0\ldots i_k}+\sum_{l=0}^k(-1)^{(l-1)(k-l+1)}a^{l,1-l}_{i_0\ldots i_l}\circ c_{i_l\ldots i_kj}\\
&+a^{k+1,-k}_{i_0\ldots i_kj}\circ c_j]\\
=& \sum_{l=1}^k(-1)^{k+l}c_{i_0\ldots \widehat{i_l}\ldots i_kj}-c_{i_0\ldots i_k}+\sum_{l=0}^k(-1)^{lk+l+1}a^{l,1-l}_{i_0\ldots i_l}\circ c_{i_l\ldots i_kj}+(-1)^ka^{k+1,-k}_{i_0\ldots i_kj}\circ c_j.
\end{split}
\end{equation*}

In short we have
\begin{equation}\label{equation: h delta_a}
\begin{split}
&(h \delta_a  c)_{i_0\ldots i_k} \\
=\sum_{l=1}^k(-1)^{k+l}c_{i_0\ldots \widehat{i_l}\ldots i_kj}-&c_{i_0\ldots i_k}+\sum_{l=0}^k(-1)^{lk+l+1}a^{l,1-l}_{i_0\ldots i_l}\circ c_{i_l\ldots i_kj}+(-1)^ka^{k+1,-k}_{i_0\ldots i_kj}\circ c_j.
\end{split}
\end{equation}

Comparing Equation \eqref{equation: delta_a h} and \eqref{equation: h delta_a} we get
$$
[\delta_a h c+h \delta_a  c]_{i_0\ldots i_k}=-c_{i_0\ldots i_k}+(-1)^ka^{k+1,-k}_{i_0\ldots i_kj}\circ c_j.
$$

Recall that $fc=c_j$ and
$$
g(fc)_{i_0\ldots i_k}=\epsilon^k a^{k+1,-k}_{i_0\ldots i_kj}\cdot c_j=(-1)^k a^{k+1,-k}_{i_0\ldots i_kj}\circ c_j
$$
hence we get the desired result
$$
[\delta_a h c+h \delta_a  c]_{i_0\ldots i_k}=-c_{i_0\ldots i_k}+g(fc)_{i_0\ldots i_k}.
$$
This finishes the proof of Proposition \ref{prop: associated sheaf locall isom to original twisted complex}.
\qed \end{proof}

The perfectness now is a direct corollary of Proposition \ref{prop: associated sheaf locall isom to original twisted complex}.

\begin{corollary}\label{coro: perfectness of sheafification}
If $\mathcal{E}=(E^{\bullet},a)$ is a twisted perfect complex, then the sheafification $\mathcal{S}^{\bullet}(\mathcal{E})$ is a perfect  complex on $(X,\mathcal{O}_X)$. In other words the sheafification functor $\mathcal{S}$ restricts to $\text{Tw}_{\text{perf}}(X)$  and gives the following dg-functor
$$
\mathcal{S}: \text{Tw}_{\text{perf}}(X)\to \text{Sh}_{\text{perf}}(X).
$$
\end{corollary}
\begin{proof}
Proposition \ref{prop: associated sheaf locall isom to original twisted complex} tells us that $\mathcal{S}^{\bullet}(\mathcal{E})|_{U_j}$ is isomorphic to $(E^{\bullet}_j,a^{0,1}_j)$  in $K(U_j)$ hence by definition it  is perfect on $U_j$. Moreover this is true for any member $U_j$ of the open cover, therefore $\mathcal{S}^{\bullet}(\mathcal{E})$ is a perfect complex of sheaves on $(X,\mathcal{O}_X)$.
\qed \end{proof}

\begin{remark}\label{remark: sheafification is both quasi-coherent and perfect}
Corollary \ref{coro: perfectness of sheafification} together with Remark \ref{remark: sheafification of twisted perfect complex is quasi-coherent} tells us that actually we have a dg-functor
$$
\mathcal{S}: \text{Tw}_{\text{perf}}(X)\to \text{Qcoh}_{\text{perf}}(X).
$$
\end{remark}

Another consequence of Proposition \ref{prop: associated sheaf locall isom to original twisted complex} is the following criterion of weak equivalence. Recall that by Definition \ref{defi: quasi-isomorphism in Tw} a closed degree zero morphism $\phi^{\bullet,-\bullet}: \mathcal{E} \to \mathcal{F}$ is called a  weak equivalence if its $(0,0)$ component $\phi^{0,0}_i: (E^{\bullet}_i,a^{0,1})\to (F^{\bullet}_i,b^{0,1})$ is a quasi-isomorphism of complexes of $\mathcal{O}_X$-modules on $U_i$ for each $i$.

\begin{corollary}\label{coro: weakly equi and sheafification}[Criterion of weak equivalence]
A degree $0$ cocycle $\phi^{\bullet,-\bullet}: \mathcal{E} \to \mathcal{F}$ in Tw$(X)$ is a weak equivalence if and only if its sheafification
$$
\mathcal{S}(\phi): \mathcal{S}(\mathcal{E})\to \mathcal{S}(\mathcal{F})
$$
is a quasi-isomorphism.
\end{corollary}
\begin{proof}
First we fix a $U_j$. It is obvious that the quasi-isomorphism
$$
f: \mathcal{S}^{\bullet}(\mathcal{E})|_{U_j} \overset{\sim}{\to} E^{\bullet}_{j}
$$
is functorial hence we have the following commutative diagram
$$
\begin{CD}
\mathcal{S}^{\bullet}(\mathcal{E})|_{U_j} @>\mathcal{S}(\phi)|_{U_j}>> \mathcal{S}^{\bullet}(\mathcal{F})|_{U_j}\\
@V\sim VV @VV \sim V\\
E^{\bullet}_{j} @>\phi^{0,0}_j >> F^{\bullet}_{j}.
\end{CD}
$$
Now the claim is obviously true.
\qed \end{proof}

\subsection{The essential surjectivity of $\mathcal{S}$}\label{subsection: essentially surjective}

\subsubsection{The twisting functor $\mathcal{T}$ and some generalities}

Remark \ref{remark: sheafification is both quasi-coherent and perfect} after Corollary \ref{coro: perfectness of sheafification} ensures that we have the dg-functor
$$
\mathcal{S}: \text{Tw}_{\text{perf}}(X)\to \text{Qcoh}_{\text{perf}}(X)
$$
which  induces an exact functor
$$
\mathcal{S}: Ho\text{Tw}_{\text{perf}}(X)\to D_{\text{perf}}(\text{Qcoh}(X)).
$$

In this subsection we will show that this  functor is essentially surjective under some mild condition. Moreover we will show that the functor
$$
\mathcal{S}: Ho\text{Tw}_{\text{perf}}(X)\to D_{\text{perf}}(X)
$$
is essentially surjective under some additional  conditions.

First we define a natural dg-functor from Sh$(X)$ to Tw$(X)$ as follows
\begin{definition}\label{defi: twisted functor}
Let $(S^{\bullet},d)$ be a complex of  $\mathcal{O}_X$-modules. We define its associated twisted complex, $\mathcal{T}(S)$,  by restricting to the $U_i$'s. In more details let $(E^{\bullet},a)=\mathcal{T}(S)$ then
$$
E^{n}_i=S^n|_{U_i}
$$
and
$$
a^{0,1}_i=d|_{U_i},~a^{1,0}_{ij}=id \text{ and } a^{k,1-k}=0 \text{ for }k\geq 2.
$$
The $\mathcal{T}$ of morphisms is defined in a similar way.

We call the dg-functor $\mathcal{T}: \text{Sh}(X)\to \text{Tw}(X)$ the \emph{twisting functor}.
\end{definition}

We would like to find the relation between the dg-functors $\mathcal{S}$ and $\mathcal{T}$. First we have the following result.

\begin{proposition}\label{prop: ST and id are quasi-isomorphic}
Let $P=(S^{\bullet},d)$ be a complex of   $\mathcal{O}_X$-modules,  the natural map
$$
\tau_P: P\to \mathcal{S}\mathcal{T}(P)
$$
is a quasi-isomorphism. Hence $\tau: id\to \mathcal{S}\mathcal{T}$ gives a natural isomorphism between functors (on the level of derived categories).
\end{proposition}
\begin{proof}
By definition   $\mathcal{S}\mathcal{T}(P)$ is the total complex of the double complex associated to $\mathcal{T}(P)$ and $\tau_P$ is given by the embedding into the $0$-th row of that double complex. Hence it is sufficient to prove that the \v{C}ech direction of the double complex is acyclic. But we know that the \v{C}ech complex (without taking global sections) is always acyclic.
\qed \end{proof}

On the other hand let $\mathcal{E}=(E,a)$ be a twisted complex, we would like to define a  closed degree $0$ morphism
$$
\gamma_{\mathcal{E}}: \mathcal{T}\mathcal{S}(\mathcal{E})\to \mathcal{E}.
$$
Actually for each $U_{i_0\ldots i_p}$ we need to construct a map
$$
(\gamma_{\mathcal{E}})^{p,-p}_{i_0\ldots i_p}: \mathcal{S}^{\bullet}(\mathcal{E})|_{U_{i_p}}\to E^{\bullet-p}_{i_0}.
$$
Recall that $\mathcal{S}^{\bullet}(\mathcal{E})=\prod_{j_0\ldots j_k} E^{\bullet-k}_{j_0}|_{U_{j_0\ldots j_k}}$, then $(\gamma_{\mathcal{E}})^{p,-p}_{i_0\ldots i_p}$ is defined to be projecting to the component $i_0\ldots i_p$. In particular $(\gamma_{\mathcal{E}})^{0,0}_j$ is the map $f$ in Proposition \ref{prop: associated sheaf locall isom to original twisted complex}. It is easy to verify that $\gamma_{\mathcal{E}}$ commutes with the differentials.

\begin{proposition}\label{prop: TS and id are isomorphic}
The map
$$
\gamma_{\mathcal{E}}: \mathcal{T}\mathcal{S}(\mathcal{E})\to \mathcal{E}
$$ is a weak equivalence.
\end{proposition}
\begin{proof}
This is a direct corollary of Proposition \ref{prop: associated sheaf locall isom to original twisted complex}.
\qed \end{proof}

\begin{remark}\label{remark: TS of twisted perfect complexes are quasi-coherent}
If $\mathcal{E}$ is a twisted perfect complex, then  $\mathcal{T}\mathcal{S}(\mathcal{E})$ is not necessarily a twisted perfect complex. Nevertheless it is easy to see that $\mathcal{T}\mathcal{S}(\mathcal{E})$ consists of complexes of quasi-coherent sheaves on each $U_i$.
\end{remark}

\begin{proposition}\label{prop: adjunction of T and S}
Let $\mathcal{E}=(E,a)$ be a twisted complex, the composition
$$
\mathcal{S}(\mathcal{E})\overset{\tau_{\mathcal{S}(\mathcal{E})}}{\longrightarrow}\mathcal{S}\mathcal{T}\mathcal{S}(\mathcal{E})\overset{\mathcal{S}(\gamma_{\mathcal{E}})}{\longrightarrow} \mathcal{S}(\mathcal{E})
$$
equals to the identity map on $\mathcal{S}(\mathcal{E})$.
\end{proposition}
\begin{proof}
The proof is just an untangling of definitions. By definition we know that
\begin{equation*}
\begin{split}
[\mathcal{S}\mathcal{T}\mathcal{S}(\mathcal{E})]^n&=\prod_{p+q=n}\prod_{i_0\ldots i_p}[(\mathcal{T}\mathcal{S}(\mathcal{E}))^q_{i_0}]|_{U_{i_0\ldots i_p}}\\
&=\prod_{p+q=n}\prod_{i_0\ldots i_p}(\prod_{s+t=q}\prod_{a_0\ldots a_s}(E^{t}_{a_0}|_{U_{a_0\ldots a_s}})|_{U_{i_0\ldots i_p}}
\end{split}
\end{equation*}
The map $\tau_{\mathcal{S}(\mathcal{E})}$ is the embedding into the $0$-th row hence it maps $\prod_{s+t=n}\prod_{a_0\ldots a_s}E^{t}_{a_0}|_{U_{a_0\ldots a_s}}$ to the $p=0, q=n$ component of the above equation, i.e. $\tau_{\mathcal{S}(\mathcal{E})}$ maps $\prod_{s+t=n}\prod_{a_0\ldots a_s}E^{t}_{a_0}|_{U_{a_0\ldots a_s}}$ to
$$
\prod_{i_0}(\prod_{s+t=n}\prod_{a_0\ldots a_s}(E^{t}_{a_0}|_{U_{a_0\ldots a_s}})|_{U_{i_0}}.
$$
Then compose with $\mathcal{S}(\gamma_{\mathcal{E}})$ and we get the identity map on $\prod_{s+t=q}E^{t}_{a_0}|_{U_{a_0\ldots a_s}}$.
\qed \end{proof}

\subsubsection{The twisted resolution and the essential surjectivity on quasi-coherent sheaves}

Let $P=(S^{\bullet},d)$ be a perfect  complex. There is no guarantee that its associated twisted complex $\mathcal{T}(P)$ is a twisted perfect  complex on the nose, even if we assume $P$ consists of quasi-coherent sheaves. Nevertheless we have a quasi-isomorphic result. First we need to introduce the following definitions.

\begin{definition}\label{defi: good space}
A locally ringed space $(U,\mathcal{O}_U)$ is called \emph{p-good} if it satisfies the following two conditions
\begin{enumerate}[1.]
\item For every perfect complex $\mathcal{P}^{\bullet}$ on $U$ which consists of quasi-coherent sheaves, there exists a strictly perfect complex $\mathcal{E}^{\bullet}$ on $U$ together with a quasi-isomorphism $u: \mathcal{E}^{\bullet}\overset{\sim}{\to}\mathcal{P}^{\bullet}$.
\item The higher cohomologies of quasi-coherent sheaves vanish, i.e. $H^k(U,\mathcal{F})=0$ for any quasi-coherent sheaf $\mathcal{F}$ on $U$ and any $k\geq 1$.
\end{enumerate}\end{definition}

\begin{remark}
The letter "p" in the term "p-good space' stands for "perfect".
\end{remark}

Then we can define p-good cover of a ringed space.

\begin{definition}[p-good cover]\label{defi: good covers}
Let $(X,\mathcal{O}_X)$ be a locally ringed space, an open cover $\{U_i\}$ of $X$ is called a \emph{p-good cover} if $(U_I,\mathcal{O}_X|_{U_I})$ is a p-good space for   any finite intersection $U_I$ of the open cover.
\end{definition}

\begin{remark}
We introduce p-good covers mainly because we need to fix a cover which works for any complex of quasi-coherent sheaves on $X$. Actually a possible alternative way is to refine the open cover and consider the refinement of twisted complexes and get a direct limit
$$
\underrightarrow{\lim}_{\text{refinement of } \{U_i\}}\text{Tw}(X,\mathcal{O}_X,\{U_i\}).
$$
Nevertheless in this paper we do not take the above approach and just stick to  a fixed p-good cover.
\end{remark}

A lot of "reasonable" ringed spaces have p-good covers. For example we have
\begin{itemize}
\item $(X,\mathcal{O}_X)$ is a separated scheme, then any affine cover is   p-good.
\item $(X,\mathcal{O}_X)$ is a complex manifold with $\mathcal{O}_X$ the sheaf of holomorphic functions. In this case a  Stein cover is   p-good.
\item $(X,\mathcal{O}_X)$ is a paracompact topological space with \emph{soft} structure sheaf $\mathcal{O}_X$. Then any contractible open cover is   p-good.
\end{itemize}
Further discussions of p-good covers   will be given in Appendix \ref{appendix: good covers}.

With the notion of p-good covers we can state and prove the following important proposition.

\begin{proposition}\label{prop: existence of resolution}[Twisted resolution, see \cite{o1985grothendieck} Proposition 1.2.3]
Assume the cover $\{U_i\}$ is p-good. Let $P=(S^{\bullet},d_S)$ be a perfect  complex which consists of quasi-coherent modules, then $\mathcal{T}(P)$ is weakly equivalent to a twisted perfect complex. More precisely there exists a twisted perfect complex  $\mathcal{E}$ together with a weak equivalence (Definition \ref{defi: quasi-isomorphism in Tw})
$$
\phi: \mathcal{E}\overset{\sim}{\to} \mathcal{T}(P).
$$
\end{proposition}
\begin{proof}
This proposition and its proof are essentially the same as Proposition 1.2.3 in \cite{o1985grothendieck}. For completeness we give the proof here in our terminology.

First we know that for each perfect complex $P=(S^{\bullet},d_S)$, there exists a strictly perfect complex  $E^{\bullet}_i$ on each $U_i$ together with a quasi-isomorphism
$$
\phi^{0,0}_i: E^{\bullet}_i\overset{\sim}{\to} S^{\bullet}|_{U_i}.
$$

Let us denote the differential of the chain complex $E^{\bullet}_i$ by $a^{0,1}_i$. Now we need to do the following two constructions:
 \begin{enumerate}
 \item Find maps $a^{k,1-k}$'s for $k \geq 1$ such that they and the $a^{0,1}_i$'s together make $E^{\bullet}_i$ a twisted complex.
 \item Extend the map $\phi^{0,0}_i$'s to get a morphism $(E^{\bullet},a)\to \mathcal{T}(P)$ in Tw$(X)$.
 \end{enumerate}
Actually we can construct the two kinds of maps simultaneously. Let $L^{\bullet}_i$ be the mapping cone of $\phi^{0,0}_i$ (So far $L^{\bullet}_i$  is not the mapping cone of any twisted complexes), which is a complex of (not necessarily locally free) sheaves on each open cover $U_i$ and we denote its differential by $A^{0,1}_i$. In fact we have
$$
L^{n}_i=\begin{matrix}E^{n+1}_i\\ \oplus \\ S^n_i\end{matrix}
$$
and
$$
A^{0,1}_i=\begin{pmatrix}-a^{0,1}_i&0\\
\phi^{0,0}_i&d_S|_{U_i}\end{pmatrix}
$$
We want to construct $A^{k,1-k}$ in $C^{k}(\mathcal{U},\text{Hom}^{1-k}(L,L))$ which make $L$ into a twisted complex. Moreover, we want $(L,A)$ to be the mapping cone of a closed degree zero morphism $\phi: \mathcal{E}\to \mathcal{T}(P)$ which extends the $\phi^{0,0}_i$. More precisely, we have the following two requirements on $A^{k,1-k}$:
\begin{enumerate}
\item $A$ satisfies the Maurer-Cartan equation
$$
\delta A+A\cdot A=0.
$$
\item We have
$$
A^{0,1}_i=\begin{pmatrix}-a^{0,1}_i&0\\
\phi^{0,0}_i&d_S|_{U_i}\end{pmatrix}, ~A^{1,0}_{ij}=\begin{pmatrix}*&0\\
*&id|_{U_{ij}}\end{pmatrix}
$$
and for $k\geq 2$, $A^{k,1-k}$ is of the form
$$
\begin{pmatrix}*~&0\\
*~&0\end{pmatrix}.
$$
\end{enumerate}

The construction involves  the previous Lemma \ref{lemma: acyclic of complex on good covers} and \ref{lemma: homotopy inverse on good covers}. For convenience we rephrase them here.

\begin{lemma}[Lemma \ref{lemma: acyclic of complex on good covers}]
Let $U$ be a subset of $X$ which satisfies $H^k(U,\mathcal{F})=0$ for any quasi-coherent sheaf $\mathcal{F}$ and any $k\geq 1$. Let $E^{\bullet}$ be a bounded above  complex of finitely generated locally free sheaves on $U$ and $F^{\bullet}$ be an acyclic complex of quasi-coherent modules on $U$, then the Hom complex $\text{Hom}^{\bullet}(E,F)$ is acyclic.
\end{lemma}

\begin{lemma}[Lemma \ref{lemma: homotopy inverse on good covers}]
Let $U$ be a subset of $X$ which satisfies $H^k(U,\mathcal{F})=0$ for any quasi-coherent sheaf $\mathcal{F}$ and any $k\geq 1$.  Suppose we have chain maps $r: E^{\bullet}\to F^{\bullet}$ and $s: G^{\bullet}\to F^{\bullet}$ between complexes of sheaves on $U$, where $E^{\bullet}$ is a bounded above complex of finitely generated locally free sheaves,  and $F^{\bullet}$ and $G^{\bullet}$ are quasi-coherent. Moreover $s$ is a quasi-isomorphism. Then $r$ factors through $s$ up to homotopy, i.e. there exists a chain map $r^{\prime}: E^{\bullet}\to G^{\bullet}$ such that $s\circ r^{\prime}$ is homotopic to $r$.
\end{lemma}

Notice that $S^n$ is quasi-coherent for each $n$, we apply Lemma \ref{lemma: homotopy inverse on good covers} to the case $U=U_{ij}, ~r=\phi^{0,0}_j: E^{\bullet}_j|_{U_{ij}}\to S^{\bullet}|_{U_{ij}}$ and $s=\phi^{0,0}_i: E^{\bullet}_i|_{U_{ij}}\to S^{\bullet}|_{U_{ij}}$ and we obtain a chain map $r^{\prime}: E^{\bullet}_j|_{U_{ij}}\to E^{\bullet}_i|_{U_{ij}}$ together with a homotopy $h: E^{\bullet}_j|_{U_{ij}}\to S^{\bullet-1}|_{U_{ij}} $ such that
$$
\phi^{0,0}_i r^{\prime}-\phi^{0,0}_j=d_S h+h a^{0,1}_j.
$$
Hence we get
$$
a^{1,0}_{ij}=r^{\prime} \text{ and } \phi^{1,-1}_{ij}=h.
$$
Moreover let
$$
A^{1,0}_{ij}=\begin{pmatrix}a^{1,0}_{ij}&0\\
-\phi^{1,-1}_{ij}&id|_{U_{ij}}\end{pmatrix}.
$$
It is clear that $A^{1,0}$ satisfies
$$
A^{1,0}\cdot A^{0,1}+A^{0,1}\cdot A^{1,0}=0.
$$

The $A^{k,1-k}$ for $k\geq 2$ are constructed by induction: Let $D$ denote the differential on $\text{Hom}^{\bullet}(L^{\bullet}_{i_k},L^{\bullet}_{i_0})$. We need to find $A^{k,1-k}_{i_0\ldots i_k}$ on $U_{i_0\ldots i_k}$ satisfying
\begin{enumerate}
\item
\begin{equation}\label{equation: twisted resolution induction step}
(-1)^{k+1} D(A^{k,1-k}_{i_0\ldots i_k})=[\delta A^{k-1,2-k}+\sum_{l=1}^{k-1}A^{l,1-l}\cdot A^{k-l,1+l-k}]_{i_0\ldots i_k}.
\end{equation}
\item $A^{k,1-k}_{i_0\ldots i_k}$ vanishes on the component $S^{\bullet}|_{U_{i_k}}$ of $L^{\bullet}_{i_k}$.
\end{enumerate}
Keep in mind that $L^n_i=E^{n+1}_i\oplus S^n_i$, Condition $2.$ is equivalent to the fact that $A^{k,1-k}_{i_0\ldots i_k}$ lies in the subcomplex $\text{Hom}^{\bullet}(E^{\bullet+1}_{i_k},L^{\bullet}_{i_0})$ of $\text{Hom}^{\bullet}(L^{\bullet}_{i_k},L^{\bullet}_{i_0})$.

It is easy to verify that $[\delta A^{1,0}+A^{1,0}\cdot A^{1,0}]_{ijk}$ lies in $\text{Hom}^{\bullet}(E^{\bullet+1}_{i_k},L^{\bullet}_{i_0})$. Hence by induction we know that the right hand side of Equation \eqref{equation: twisted resolution induction step}, $[\delta A^{k-1,2-k}+\sum_{l=1}^{k-1}A^{l,1-l}\cdot A^{k-l,1+l-k}]_{i_0\ldots i_k}$, lies in $\text{Hom}^{\bullet}(E^{\bullet+1}_{i_k},L^{\bullet}_{i_0})$ for $k\geq 2$. Also by induction we can show that it is a cocycle under the differential $D$. By Lemma \ref{lemma: acyclic of complex on good covers} we know that $\text{Hom}^{\bullet}(E^{\bullet+1}_{i_k},L^{\bullet}_{i_0})$ is acyclic, hence the $A^{k,1-k}_{i_0\ldots i_k}$ in $\text{Hom}^{\bullet}(E^{\bullet+1}_{i_k},L^{\bullet}_{i_0})$ which satisfies Equation \eqref{equation: twisted resolution induction step} exists. By induction we construct the desired $(L,A)$.
\qed \end{proof}

With the help of Proposition \ref{prop: existence of resolution} we can prove the essential surjectivity of the sheafification functor $\mathcal{S}$.
\begin{corollary}\label{coro: essentially surjective}[Essential surjectivity]
If the cover $\{U_i\}$ is p-good, then the sheafification functor
$$
\mathcal{S}: \text{Tw}_{\text{perf}}(X)\to \text{Qcoh}_{\text{perf}}(X)
$$
induces an essentially surjective functor
$$
\mathcal{S}: Ho\text{Tw}_{\text{perf}}(X)\to D_{\text{perf}}(\text{Qcoh}(X)).
$$

\end{corollary}
\begin{proof}
Let $P=(S^{\bullet},d)$ be an object in $\text{Qcoh}_{\text{perf}}(X)$.  Consider the associated twisted complex $\mathcal{T}(P)$, by Proposition \ref{prop: existence of resolution} there exists a twisted complex $\mathcal{E}$ together with a weak equivalence
$$
\phi: \mathcal{E}\overset{\sim}{\to} \mathcal{T}(P).
$$

Then by Corollary \ref{coro: weakly equi and sheafification} we get a quasi-isomorphism
$$
\mathcal{S}(\phi): \mathcal{S}(\mathcal{E})\overset{\sim}{\to}\mathcal{S} \mathcal{T}(P).
$$
On the other hand Proposition \ref{prop: ST and id are quasi-isomorphic} provides us another quasi-isomorphism
$$
\tau_P: P\overset{\sim}{\to} \mathcal{S} \mathcal{T}(P).
$$
Therefore  $\mathcal{S}(\mathcal{E})$ is quasi-isomorphic to $P$, which finishes the proof of Corollary \ref{coro: essentially surjective}.
\qed \end{proof}

\subsubsection{Essential surjectivity on complexes of $\mathcal{O}_X$-modules}\label{subsubsection: Essential surjectivity on complexes of A-modules}
Now we want to show that the following functor
$$
\mathcal{S}: Ho\text{Tw}_{\text{perf}}(X)\to D_{\text{perf}}(X)
$$
is essentially surjective. For this we need the following additional condition on the ringed space $(X,\mathcal{O}_X)$.

\begin{definition}\label{defi: perfect-equivalent condition}
We say a locally ringed space $(X,\mathcal{O}_X)$ satisfies the \emph{perfect-equivalent condition} if the natural map
$$
D_{\text{perf}}(\text{Qcoh}(X))\to D_{\text{perf}}(X)
$$
is an equivalence.
\end{definition}

Further discussions of perfect-equivalent  condition will be given in Appendix \ref{appendix: quasi-coherent, pseudo-coherent complexes and coherent sheaves}. In particular we can show that any quasi-compact and semi-separated scheme or any Noetherian scheme satisfies the perfect-equivalent condition.

With Definition \ref{defi: perfect-equivalent condition} we have the following result.

\begin{corollary}\label{coro: essentially surjective on O-modules}[Essential surjectivity]
If $X$ satisfies the perfect-equivalent  condition and the cover $\{U_i\}$ is p-good, then the functor
$$
\mathcal{S}: Ho\text{Tw}_{\text{perf}}(X)\to D_{\text{perf}}(X)
$$
is essentially surjective.
\end{corollary}
\begin{proof}
It is a direct corollary of Corollary \ref{coro: essentially surjective} and Definition \ref{defi: perfect-equivalent condition}.
\qed \end{proof}

\subsection{The fully-faithfulness  of the sheafification functor}\label{subsection: fully faithfull}
\subsubsection{Fully faithful on complexes of quasi-coherent sheaves}
We want to show that the sheafification functor $\mathcal{S}$ induces a fully faithful functor
$$
\mathcal{S}: Ho\text{Tw}_{\text{perf}}(X)\to D_{\text{perf}}(\text{Qcoh}(X)).
$$

First we have the following proposition.

\begin{proposition}\label{prop: invertible in hoTw and quasi-isomorphism}
Let the cover $\{U_i\}$ satisfy $H^k(U_i,\mathcal{F})=0$ for any $i$, any quasi-coherent sheaf $\mathcal{F}$ on $U_i$ and any $k\geq 1$.  If $\mathcal{E}$ and $\mathcal{F}$ are both in the subcategory $\text{Tw}_{\text{perf}}(X)$, then $\mathcal{S}(\phi): \mathcal{S}(E)\to \mathcal{S}(\mathcal{F})$ is a quasi-isomorphism if and only if    $\phi: \mathcal{E}\to \mathcal{F}$ is invertible in $Ho\text{Tw}_{\text{perf}}(X)$.
\end{proposition}
\begin{proof}
We first use Proposition \ref{prop: homotopy invertible morphisms}, which claims that  $\phi: \mathcal{E}\to \mathcal{F}$ is invertible in $Ho\text{Tw}(X)$ if and only if $\phi$ is a weak equivalence. Moreover Corollary \ref{coro: weakly equi and sheafification} tells us $\phi$ is a weak equivalence if and only if $\mathcal{S}(\phi): \mathcal{S}(E)\to \mathcal{S}(\mathcal{F})$ is a quasi-isomorphism, hence we get the result.
\qed \end{proof}

Now we are about to prove the full-faithfulness of the functor $\mathcal{S}$. We divide the proof into several steps and first we have the following  lemma.

\begin{lemma}\label{lemma: fullness}[Fullness]
If the cover $\{U_i\}$ is p-good, then the functor $\mathcal{S}: Ho\text{Tw}_{\text{perf}}(X)\to D_{\text{perf}}(\text{Qcoh}(X))$ is full.
\end{lemma}
\begin{proof}
Let $\mathcal{A} $ and $\mathcal{B}$ be two objects in $\text{Tw}_{\text{perf}}(X)$. A morphism $\mathcal{S}(\mathcal{A}) \to \mathcal{S}(\mathcal{B})  $ in $D_{\text{perf}}(\text{Qcoh}(X))$ can be written as
$$
\begin{tikzcd}[column sep=small]
&\mathcal{P}  \arrow {dl}{\sim}[swap]{\mu}  \arrow {dr}{\varphi}&\\
\mathcal{S}(\mathcal{A}) & &\mathcal{S}(\mathcal{B}).
\end{tikzcd}
$$
Applying $\mathcal{T}$ we get
$$
\begin{tikzcd}[column sep=small]
&\mathcal{T}(\mathcal{P})  \arrow {dl}{\sim}[swap]{\mathcal{T}(\mu)}  \arrow {dr}{\mathcal{T}(\varphi)}&\\
\mathcal{T}\mathcal{S}(\mathcal{A}) & &\mathcal{T}\mathcal{S}(\mathcal{B}).
\end{tikzcd}
$$

$\mathcal{P}$ is a perfect complex since it is quasi-isomorphic to $\mathcal{S}(\mathcal{A})$. Then by Proposition \ref{prop: existence of resolution} there exists a resolution $\phi:\mathcal{E}\overset{\sim}{\to}\mathcal{T}(\mathcal{P})$ and hence
$$
\begin{tikzcd}[column sep=small]
&\mathcal{E}  \arrow {dl}{\sim}[swap]{\mathcal{T}(\mu)\circ \phi}  \arrow {dr}{\mathcal{T}(\varphi)\circ \phi}&\\
\mathcal{T}\mathcal{S}(\mathcal{A}) & &\mathcal{T}\mathcal{S}(\mathcal{B}).
\end{tikzcd}
$$
Compose with $\gamma_{\mathcal{A}}: \mathcal{T}\mathcal{S}(\mathcal{A})\to \mathcal{A}$ and $\gamma_{\mathcal{B}}: \mathcal{T}\mathcal{S}(\mathcal{B})\to \mathcal{B}$ we get
$$
\begin{tikzcd}[column sep=small]
&\mathcal{E}  \arrow {dl}{\sim}[swap]{\gamma_{\mathcal{A}}\circ\mathcal{T}(\mu)\circ \phi}  \arrow {dr}{\gamma_{\mathcal{B}}\circ\mathcal{T}(\varphi)\circ \phi}&\\
\mathcal{A} & & \mathcal{B}.
\end{tikzcd}
$$
The left map $\gamma_{\mathcal{A}}\circ\mathcal{T}(\mu)\circ \phi$ is a weak equivalence between twisted perfect complexes hence by Proposition \ref{prop: homotopy invertible morphisms} it is invertible up to homotopy. Let $\theta: \mathcal{A}\to \mathcal{B}$ be the composition
$$
\theta:=\gamma_{\mathcal{B}}\circ\mathcal{T}(\varphi)\circ \phi\circ (\gamma_{\mathcal{A}}\circ\mathcal{T}(\mu)\circ \phi)^{-1}.
$$

It is clear that  $\mathcal{S}(\theta)$ equals to $\varphi\circ (\mu)^{-1}$ in the derived category. We know that $\mathcal{S}$ is full.
\qed \end{proof}

\begin{lemma}\label{lemma: faithfulness}[Faithfulness]
If the cover $\{U_i\}$ is p-good, then the functor $\mathcal{S}: Ho\text{Tw}_{\text{perf}}(X)\to D_{\text{perf}}(\text{Qcoh}(X))$ is faithful.
\end{lemma}
\begin{proof}
Let $\theta: \mathcal{A}\to \mathcal{B}$ be a morphism between twisted perfect complexes such that $\mathcal{S}(\theta)=0$ in the derived category. Then, by definition, there is a complex $\mathcal{P}$ together with a quasi-isomorphism
$$
\mu: \mathcal{P}\to \mathcal{S}(\mathcal{A})
$$
such that $\mathcal{S}(\theta)\circ \mu$ is homotopic to $0$. It follows that
$$
\mathcal{T}(\mathcal{P})\overset{\mathcal{T}(\mu)}{\longrightarrow} \mathcal{T}\mathcal{S}(\mathcal{A})\overset{\mathcal{T}(\mathcal{S}(\theta))}{\longrightarrow} \mathcal{T}\mathcal{S}(\mathcal{B})
$$
is homotopic to $0$.

On the other hand we have the following commutative diagram
$$
\begin{CD}
\mathcal{T}(\mathcal{P})@>\mathcal{T}(\mu)>>\mathcal{T}(\mathcal{S}(\mathcal{A})) @>\mathcal{T}(\mathcal{S}(\theta))>>\mathcal{T}(\mathcal{S}(\mathcal{B}))\\
@. @V\sim V\gamma_{\mathcal{A}}V @V\gamma_{\mathcal{B}}V\sim V\\
@. \mathcal{A}@>\theta >>\mathcal{B}.
\end{CD}
$$
hence $\theta\circ \gamma_{\mathcal{A}}\circ \mathcal{T}(\mu)$ is homotopic to $0$ and so is $\theta\circ \gamma_{\mathcal{A}}\circ \mathcal{T}(\mu)\circ \phi$, where $\phi: \mathcal{E}\to \mathcal{T}(\mathcal{P})$ is as in the proof of Lemma \ref{lemma: fullness}. From this we conclude that $\theta$ is homotopic to $0$ because $\gamma_{\mathcal{A}}\circ \mathcal{T}(\mu)\circ \phi$ is invertible up to homotopy.
\qed \end{proof}

\begin{corollary}\label{coro: fully faithful}
If the cover $\{U_i\}$ is p-good, then the functor $\mathcal{S}: Ho\text{Tw}_{\text{perf}}(X)\to D_{\text{perf}}(\text{Qcoh}(X))$ is fully faithful.
\end{corollary}
\begin{proof}
It is a immediate corollary of Lemma \ref{lemma: fullness} and Lemma \ref{lemma: faithfulness}.
\qed \end{proof}

\begin{remark}\label{remark: higher morphisms are good}
The great advantage of twisted complexes is that we have more flexibility on morphisms. For example when $(X,\mathcal{O}_X)$ is a projective scheme, then it is well-known that any perfect complex on $X$ is strictly perfect. In other words let $L(X)$ be the dg-category of two-side bounded complexes of finitely generated locally free sheaves on $X$. Then the natural functor $HoL(X)\to  D_{\text{perf}}(\text{Qcoh}(X))$ is essentially surjective but not necessarily fully faithful.

In fact let $\mathcal{E}$ and $\mathcal{F}$ be two objects in $L(X)$  and
$\phi: \mathcal{E}\overset{\sim}{\to} \mathcal{F}$ be a quasi-isomorphism. Then in general $\phi$ does not have an inverse in $HoL(X)$. Nevertheless the inverse of $\phi$  exists in $Ho\text{Tw}_{\text{perf}}(X)$ if we consider $\mathcal{E}$ and $\mathcal{F}$ as twisted perfect complexes through the twisting functor $\mathcal{T}$ and the cover is p-good.
\end{remark}

Now we can state the main theorem of this paper.

\begin{theorem}\label{thm: dg-enhancement}[dg-enhancement, see Theorem \ref{thm: dg-enhancement in the introduction} in the Introduction]
If the cover $\{U_i\}$ is p-good, then the sheafification functor $\mathcal{S}: \text{Tw}_{\text{perf}}(X)\to \text{Qcoh}_{\text{perf}}(X)$ gives an equivalence of categories
$$
\mathcal{S}: Ho\text{Tw}_{\text{perf}}(X)\to D_{\text{perf}}(\text{Qcoh}(X))
$$
\end{theorem}
\begin{proof}
This is a immediate consequence of Corollary \ref{coro: essentially surjective} and Corollary \ref{coro: fully faithful}.
\qed \end{proof}

\begin{example}\label{eg: dg-enhancement of quasi-coherent}
We have the following cases which we can apply Theorem \ref{thm: dg-enhancement}. In fact we only need to verify that the following spaces have p-good  covers. For more discussion on p-good covers see Appendix \ref{appendix: good covers}.
\begin{itemize}
\item Let $(X,\mathcal{O}_X)$ be a  separated scheme and $\{U_i\}$ be an affine cover, then we have  an equivalence of categories $\mathcal{S}: Ho\text{Tw}_{\text{perf}}(X)\to D_{\text{perf}}(\text{Qcoh}(X))$.

\item Let $X$ be a complex manifold with the structure sheaf of holomorphic functions, then we have an equivalence of categories $\mathcal{S}: Ho\text{Tw}_{\text{perf}}(X)\to D_{\text{perf}}(\text{Qcoh}(X))$.

\item Let $X$ be a smooth manifold with the structure sheaf of smooth functions, then we have an equivalence of categories $\mathcal{S}: Ho\text{Tw}_{\text{perf}}(X)\to D_{\text{perf}}(\text{Qcoh}(X))$.
\end{itemize}
\end{example}

\subsubsection{Fully faithful on complexes of $\mathcal{O}_X$-modules}
Similar to the discussion in Section \ref{subsubsection: Essential surjectivity on complexes of A-modules}, we can add certain conditions on $X$ and get the fully faithfulness on perfect complexes of arbitrary $\mathcal{O}_X$-modules.

\begin{corollary}\label{coro: fully faithful on non-quasi-coherent sheaves}[Fully faithful]
   If $X$ satisfies the perfect-equivalent  condition and the cover $\{U_i\}$ is p-good, then the functor
$$
\mathcal{S}: Ho\text{Tw}_{\text{perf}}(X)\to D_{\text{perf}}(X)
$$
is fully faithful.
\end{corollary}
\begin{proof}
It is a direct consequence of Corollary \ref{coro: fully faithful} and the  perfect-equivalent  condition (Definition \ref{defi: perfect-equivalent condition}).
\qed \end{proof}

\begin{theorem}\label{thm: dg-enhancement of A-modules}[dg-enhancement, see Theorem \ref{thm: dg-enhancement of A-modules in the introduction} in the Introduction]
If $X$ satisfies the perfect-equivalent  condition and  the cover $\{U_i\}$ is p-good, then the sheafification functor $\mathcal{S}: \text{Tw}_{\text{perf}}(X)\to \text{Sh}_{\text{perf}}(X)$ gives an  equivalence of categories
$$
\mathcal{S}: Ho\text{Tw}_{\text{perf}}(X)\to D_{\text{perf}}(X)
$$
\end{theorem}
\begin{proof}
This is a immediate consequence of Corollary \ref{coro: essentially surjective on O-modules} and Corollary \ref{coro: fully faithful on non-quasi-coherent sheaves}.
\qed \end{proof}

\begin{example}\label{eg: dg-enhancement of A-modules}
The application of Theorem \ref{thm: dg-enhancement of A-modules} is more restrictive than Theorem \ref{thm: dg-enhancement} since we need to verify the perfect-equivalent condition. Nevertheless it contains the following important cases:
 Let $(X,\mathcal{O}_X)$ be a quasi-compact
and semi-separated or Noetherian scheme and $\{U_i\}$ be an affine cover, then we have  an  equivalence of categories $\mathcal{S}: Ho\text{Tw}_{\text{perf}}(X)\to D_{\text{perf}}(X)$. See Appendix \ref{appendix: quasi-coherent, pseudo-coherent complexes and coherent sheaves} Corollary \ref{coro: spaces satisfy the perfect-equivalent condition}.
\end{example}

\begin{remark}\label{remark: relation with Lunts-Schnurer}
As we mentioned in Remark \ref{remark: relation with Lunts-Schnurer introduction} in the introduction, the twisted complexes is very similar to the \v{C}ech enhancement introduced in \cite{lunts2014new}. In fact for a complex of sheaves $\mathcal{E}$ on $X$, we could see that $\mathcal{S}\mathcal{T}(\mathcal{E})$ is almost the same as the $\mathcal{E}_{\supset}$ in \cite{lunts2014new} Section 3.2.3. Nevertheless, our twisted complexes and the \v{C}ech enhancement in \cite{lunts2014new} have the following two main differences.
\begin{enumerate}
\item Twisted complexes allow twists ($a^{i,1-i}$'s) hence we could find resolutions of non-strictly perfect complexes on non-GSP schemes.

\item We do not have an order on the open subsets and we do not assume the open cover is finite.

\item We do not us the pushforward $i_{!}$ hence we do not consider the $\mathcal{E}^{\supset}$ as in \cite{lunts2014new} Section 3.2.3.
\end{enumerate}
\end{remark}

\section{APPLICATIONS OF TWISTED COMPLEXES}\label{section: applications}
Twisted complexes have various applications. For example in \cite{o1985grothendieck} twisted complexes are used to formulate and prove a Grothendieck-Riemann-Roch theorem for perfect complexes and in \cite{gillet1986k} they are used to compute the higher algebraic K-theory of schemes.

\begin{remark}
Neither of the above works uses the fact that twisted perfect complexes is a dg-enhancement of perfect complexes.
\end{remark}

In this paper we talk about the application of twisted complexes in descent theory. It is well-known that one of the drawbacks of derived categories is that they do not satisfy descent. In more details, let $X$ be a scheme and $U$, $V$ be an open cover of $X$, then we have derived categories $D_{\text{perf}}(X)$, $D_{\text{perf}}(U)$, $D_{\text{perf}}(V)$, and $D_{\text{perf}}(U\cap V)$. Moreover we have the fiber product of categories $D_{\text{perf}}(U)\times _{D_{\text{perf}}(U\cap V)}D_{\text{perf}}(V)$. However the natural functor
$$
D_{\text{perf}}(X)\to D_{\text{perf}}(U)\times _{D_{\text{perf}}(U\cap V)}D_{\text{perf}}(V)
$$
is not an equivalence even in the case that $X=\mathbb{P}^1$ and $U$, $V$ are the upper and lower hemispheres. See \cite{bertrand2011lectures} Section 2.2 (d) for more details.

This problem can be solved in the framework of dg-categories. In fact Tabuada in \cite{tabuada2010homotopy} gives an explicit construction of path object in dg-categories, which leads to the following definition of homotopy fiber product of dg-categories.

\begin{definition}\label{defi: homotopy fiber product}[\cite{ben2013milnor} Section 4]
Let $A$, $B$, $C$ be dg-categories and $\phi: A\to C$, $\theta: B\to C$ be dg-functors. Then the homotopy fiber product $A\times^h_C B$ is a dg-category with objects
\begin{equation*}
\begin{split}
ob(A\times^h_C B)&=\{M,N,f|M\in ob(A),N\in ob(B),\\
&f:\phi(M)\to \theta(N)  \text{ closed of degree }0 \text{ and invertible in }H^0(C)\}.
\end{split}
\end{equation*}

The degree $k$ morphisms between $(M_1,N_1,f_1)$ and $(M_2,N_2,f_2)$ are given by
$$
(\mu,\nu,\tau)\in A^k(M_1,M_2)\oplus B^k(N_1,N_2)\oplus C^{k-1}(\phi(M_1),\theta(N_2))
$$
with composition given by
$$
(\mu^{\prime},\nu^{\prime},\tau^{\prime})(\mu,\nu,\tau)=(\mu^{\prime}\mu,\nu^{\prime}\nu,\tau^{\prime}\phi(\mu)+\theta(\nu^{\prime})\tau).
$$

The differential on the morphisms is given by
$$
d(\mu,\nu,\tau)=(d\mu,d\nu,d\tau+f_2\phi(\mu)-(-1)^k\theta(\nu)f_1).
$$
\end{definition}

\begin{remark}
We should mention that in \cite{bertrand2011lectures} Section 5.3 To\"{e}n uses the injective enhancement $L_{pe}(X)$ and claims that
 $$
 L_{pe}(X)\overset{\sim}\to L_{pe}(U)\times^h_{L_{pe}(U\cap V)} L_{pe}(V).
 $$

 Moreover in \cite{ben2013milnor} the authors use the cohesive modules as another dg-enhancements and prove that they have the descent property.
 \end{remark}

Now we move on to the descent problem of twisted perfect complexes. Let  $X$ be a separated scheme and $X=U\cup V$ be two open subsets. For simplicity let us consider the   case that $U$ and $V$ are affine. Then $U \cap V$ is  affine too. Moreover, $\{U ,V\}$ gives an affine (hence p-good) open cover of $X$ and we have $\text{Tw}_{\text{perf}}(X,\mathcal{O}_X,\{U ,V\})$.

It is clear that  $\text{Tw}_{\text{perf}}(U,\mathcal{O}_U,\{U \})$ is exactly the dg-category of strictly perfect complexes on $U$. The same assertion holds for  $\text{Tw}_{\text{perf}}(V,\mathcal{O}_V,\{V \})$ and $\text{Tw}_{\text{perf}}(U\cap V,\mathcal{O}_{U\cap V},\{U \cap V\})$. There are natural dg-functors
$$
\phi: \text{Tw}_{\text{perf}}(U,\mathcal{O}_U,\{U \})\to \text{Tw}_{\text{perf}}(U\cap V,\mathcal{O}_{U\cap V},\{U \cap V\})
$$
and
$$
\theta: \text{Tw}_{\text{perf}}(V,\mathcal{O}_V,\{V\}) \to \text{Tw}_{\text{perf}}(U\cap V,\mathcal{O}_{U\cap V},\{U \cap V\})
$$
given by restriction.

We omit the open covers and structure rings in the notation of the twisted perfect complexes and we have the following descent property.

\begin{proposition}\label{prop: descent of twisted perfect complexes}
Let $X$, $U$, $V$ be as above, then we have a quasi-equivalence of dg-categories
$$
\text{Tw}_{\text{perf}}(X)\overset{\sim}{\to} \text{Tw}_{\text{perf}}(U)\times^h_{\text{Tw}_{\text{perf}}(U\cap V)}\text{Tw}_{\text{perf}}(V).
$$
\end{proposition}
\begin{proof}
The main part of the proof is to untangle the definition of homotopy fiber product. Let $\mathcal{E}$ be an object in $\text{Tw}_{\text{perf}}(X)$. Then it gives $E^{\bullet}_U$ on $U $ and $E^{\bullet}_V$ on $V$. 

It is clear  $E^{\bullet}_U$  together with  $a^{0,1}_U$  give an object in $\text{Tw}_{\text{perf}}(U)$ and we denote it by $\mathcal{M}$. Similarly $E^{\bullet}_V$ together with $a^{0,1}_V$ give an object $\mathcal{N}$ in $\text{Tw}_{\text{perf}}(V)$. Moreover the map $a^{1,0}_{VU}$ gives the morphism
$$
f: \mathcal{M}\to \mathcal{N}
$$
and it is homotopic invertible since the $a$'s satisfies the Maurer-Cartan equation and has the non-degenerate property.

Hence we get a dg-functor
$$
R: \text{Tw}_{\text{perf}}(X)\to \text{Tw}_{\text{perf}}(U)\times^h_{\text{Tw}_{\text{perf}}(U\cap V)}\text{Tw}_{\text{perf}}(V).
$$
It is clear that $R$ is essentially surjective. By the same method as in the proof of Proposition \ref{prop: existence of resolution} we can prove it is also quasi-fully faithful.
\qed \end{proof}

\begin{remark}
The same idea works for the general case where $U$ and $V$ are not affine. Nevertheless we need an explicit construction of homotopy limit of dg-categories and this topic will be treated in another paper.
\end{remark}

\section{FURTHER TOPICS}\label{section: further topics}

\subsection{Twisted coherent complexes}\label{subsection: twisted coherent complexes}
In this subsection we consider a variation of twisted perfect complex, where the two-side bounded complexes are replaced by bounded above complexes. We omit most of the proofs since they are the same as the corresponding proofs for the twisted perfect complexes.

\subsubsection{The derived category of bounded above coherent complexes}

First we review the relevant derived categories. We  have a definition of coherent complex.

\begin{definition}\label{defi: coherent complex and strictly coherent complex}
Let $(X,\mathcal{O}_X)$ be a separated, Noetherian scheme. A complex $\mathcal{S}^{\bullet}$ of $\mathcal{O}_X$-modules is bounded above and coherent if for any point $x\in X$, there exists an open neighborhood $U$ of $x$ and a bounded above complex of finite rank, locally free sheaves  $\mathcal{E}^{\bullet}_{U}$ on $U$ such that the restriction $\mathcal{S}^{\bullet}|_U$ is isomorphic to $\mathcal{E}^{\bullet}_{U}$ in $D(\mathcal{O}_X|_U-\text{mod})$, the derived category of sheaves of $\mathcal{O}_X$-modules on $U$.
\end{definition}

\begin{remark}\label{remark: coherent and pseudo-coherent compare}
If $X$ is not a separated Noetherian scheme then the category of bounded above coherent complexes does not behave well. In fact a more standard notion is the \emph{pseudo-coherent} complex on a ringed space, see \cite{berthelot1966seminaire} Expos\'{e} I or \cite{thomason1990higher} Section 2. Nevertheless, pseudo-coherent coincides with our definition of coherent if $X$ is a Noetherian scheme as shown in Appendix \ref{appendix: quasi-coherent, pseudo-coherent complexes and coherent sheaves}. In this paper we will stick to the above definition of coherent complex.
\end{remark}

In this subsection we always assume $X$ is a separated Noetherian scheme.

We consider the following  categories.

\begin{definition}\label{defi: derived cat of coherent complex}
  Let $\text{Sh}^-_{\text{coh}}(X)$ be the full dg-subcategory of Sh$(X)$ which consists of bounded above coherent complexes on $X$.

Similarly we have $K^-_{\text{coh}}(X)$, $D^-_{\text{coh}}(X)$, $K^-_{\text{coh}}(\text{Qcoh}(X))$, and $D^-_{\text{coh}}(\text{Qcoh}(X))$
\end{definition}

\subsubsection{Twisted coherent complexes}
We have the following definition which is similar to Definition \ref{defi: twisted perfect complex}.
\begin{definition}\label{defi: twisted coherent complex}
A \emph{twisted coherent complex} $\mathcal{E}=(E^{\bullet},a)$ is the same as twisted complex except that $E^{\bullet}$ are   bounded above graded finitely generated locally free   $\mathcal{O}_X$-modules.

The twisted coherent complexes form a dg-category and we denote it by $\text{Tw}^-_{\text{coh}}(X,\mathcal{O}_X, \{U_i\})$ or simply $\text{Tw}^-_{\text{coh}}(X)$. Obviously $\text{Tw}^-_{\text{coh}}(X)$ is a full dg-subcategory of Tw$(X)$ while $\text{Tw}_{\text{perf}}(X)$ is a full dg-subcategory of $\text{Tw}^-_{\text{coh}}(X)$.
\end{definition}

The differential $\delta_a$, shift functor, mapping cone and weak equivalence as in Section \ref{subsection: pre-triangulated structure} and \ref{subsection: weak equivalence} can be defined on $\text{Tw}^-_{\text{coh}}(X)$ without any change. Moreover we have the same result as in Proposition \ref{prop: homotopy invertible morphisms}

\begin{proposition}\label{prop: homotopy invertible morphisms for coherent}
Let the cover $\{U_i\}$  satisfy $H^k(U_i,\mathcal{F})=0$ for any $i$ , any quasi-coherent sheaf $\mathcal{F}$ on $U_i$ and any $k\geq 0$. If $\mathcal{E}$ and $\mathcal{F}$ are both in the subcategory $\text{Tw}^-_{\text{coh}}(X)$, then a closed degree zero morphism $\phi$ between twisted complexes $\mathcal{E}$ and $\mathcal{F}$ is  a weak equivalence  if and only if $\phi$ is invertible in the homotopy category $\text{HoTw}_{\text{coh}}(X)$.
\end{proposition}
\begin{proof}
Notice that in the proof of Proposition \ref{prop: homotopy invertible morphisms} we do not use the boundedness of the complexes hence the same proof works for $\text{HoTw}_{\text{coh}}(X)$.
\qed \end{proof}

\subsubsection{The sheafification functor on twisted coherent complexes}
We wish to restrict the sheafification functor in  Definition \ref{defi: sheafification functor} to twisted coherent complexes and get a dg-functor
$$
\mathcal{S}: \text{Tw}^-_{\text{coh}}(X)\to \text{Sh}(X).
$$
and  we want to a result which is similar to Remark \ref{remark: sheafification of twisted perfect complex is quasi-coherent}, i.e. we want the dg-functor $\mathcal{S}$  maps $\text{Tw}^-_{\text{coh}}(X)$ to complexes of quasi-coherent sheaves. However there is a serious problem here. Recall the in Equation \eqref{equation: sheafification functor: alternative} of the definition of $\mathcal{S}$ we have
$$
\mathcal{S}^n(\mathcal{E})=\prod_{p+q=n}E^q_{i_0}|_{U_{i_0\ldots i_p}}.
$$
Now $\mathcal{E}=(E^{\bullet}_i,a)$ is a twisted coherent complex hence $E^{\bullet}_i$ is bounded above. Therefore $\prod_{p+q=n}E^q_{i_0}|_{U_{i_0\ldots i_p}}$ is an infinite product. The problem is that the category Qcoh$(X)$ does not have infinite direct products for general $X$, and even when it has, the infinite direct product in Qcoh$(X)$ is not the same as the product in the larger category Sh$(X)$.

To solve this problem we have to slightly modify the definition of $\mathcal{S}$. First we introduce the following definition.

\begin{definition}\label{defi: twisted q-coherent complex}
A \emph{twisted quasi-coherent complex} $\mathcal{E}=(E^{\bullet},a)$ is the same as twisted complex except that $E^{\bullet}$ are     graded quasi-coherent  $\mathcal{O}_X$-modules.

The twisted quasi-coherent complexes form a dg-category and we denote it by $\text{Tw}_{\text{qcoh}}(X,\mathcal{O}_X, \{U_i\})$ or simply $\text{Tw}_{\text{qcoh}}(X)$. Obviously $\text{Tw}_{\text{qcoh}}(X)$ is a full dg-subcategory of Tw$(X)$ while $\text{Tw}_{\text{perf}}(X)$ and $\text{Tw}^-_{\text{coh}}(X)$ are full dg-subcategories of $\text{Tw}_{\text{qcoh}}(X)$.
\end{definition}

Before defining the sheafification functor we need the following lemma.

\begin{lemma}\label{lemma: infinite limit of quasi-coherent sheaves}
Let $X$ be a quasi-compact and quasi-separated scheme, then the category Qcoh$(X)$ has all limits.
\end{lemma}
\begin{proof}
See \cite{thomason1990higher} Lemma B.12.
\qed \end{proof}

Then we define the sheafification functor  in a slightly modified way and we will call it $\widetilde{\mathcal{S}}$.
\begin{definition}\label{defi: sheafification functor for q-coherent}
Let $X$ be  a separated Noetherian (hence  quasi-compact and quasi-separated) scheme. The definition of $\widetilde{\mathcal{S}}: \text{Tw}_{\text{qcoh}}(X)\to \text{Qcoh}(X)$ is the same as that of $\mathcal{S}$ in Definition \ref{defi: sheafification functor} except that in the equation
$$
\widetilde{\mathcal{S}}^n(\mathcal{E})=\prod_{p+q=n}E^q_{i_0}|_{U_{i_0\ldots i_p}}
$$
we take the direct  product in Qcoh$(X)$. By Lemma \ref{lemma: infinite limit of quasi-coherent sheaves}, $\widetilde{\mathcal{S}}$ is well-defined.
\end{definition}

\begin{remark}
 $\widetilde{\mathcal{S}}$ coincides with $\mathcal{S}$ if restricted to $\text{Tw}_{\text{perf}}(X)$ since in this case the product $\prod_{p+q=n}E^q_{i_0}|_{U_{i_0\ldots i_p}}$ is finite and the product in Qcoh$(X)$ coincides with that in Sh$(X)$.
\end{remark}

Keep in mind that Proposition \ref{prop: associated sheaf locall isom to original twisted complex} works for any twisted complexes, hence it works for twisted coherent complexes. Moreover we also have the same result as in Corollary  \ref{coro: perfectness of sheafification}
\begin{proposition}\label{prop: coherence of sheafification}
If $\mathcal{E}=(E^{\bullet},a)$ is a twisted coherent complex, then the sheafification $\widetilde{\mathcal{S}}^{\bullet}(\mathcal{E})$ is a  coherent  complex of sheaves on $(X,\mathcal{O}_X)$. In other words the sheafification functor $\widetilde{\mathcal{S}}$ restricts to   $\text{Tw}^-_{\text{coh}}(X)$  and gives the following dg-functor
$$
\widetilde{\mathcal{S}}: \text{Tw}^-_{\text{coh}}(X)\to \text{Qcoh}^-_{\text{coh}}(X).
$$
\end{proposition}
\begin{proof}
The proof is the same as that of Corollary \ref{coro: perfectness of sheafification}.
\qed \end{proof}

\subsubsection{The essential surjectivity in the coherent case}
Similar to the discussion in Section \ref{subsection: essentially surjective}, the dg-functor
$$
\widetilde{\mathcal{S}}: \text{Tw}^-_{\text{coh}}(X)\to \text{Qcoh}^-_{\text{coh}}(X)
$$
 induces a functor
$$
\widetilde{\mathcal{S}}: Ho\text{Tw}^-_{\text{coh}}(X)\to D^-_{\text{coh}}(\text{Qcoh}(X)).
$$
In this subsection we will show that this functor is essentially surjective under some mild condition. Moreover we will show that the functor
$$
\widetilde{\mathcal{S}}: Ho\text{Tw}^-_{\text{coh}}(X)\to D^-_{\text{coh}}(X)
$$
is essentially surjective under some additional  conditions.

First we have the following definitions which are similar to Definition \ref{defi: good space} and \ref{defi: good covers}.
\begin{definition}\label{defi: c-good space}
 A locally ringed space $(U,\mathcal{O}_U)$ is called \emph{c-good} if it satisfies
 \begin{enumerate}[a.]
\item For every coherent complex $\mathcal{C}^{\bullet}$ on $U$ which consists of quasi-coherent sheaves, there exists a   bounded above complex of finitely generated locally free sheaves $\mathcal{E}^{\bullet}$ together with a quasi-isomorphism $v: \mathcal{E}^{\bullet}\overset{\sim}{\to}\mathcal{C}^{\bullet}$.

\item The higher cohomologies of quasi-coherent sheaves vanish, i.e. $H^k(U,\mathcal{F})=0$ for any quasi-coherent sheaf $\mathcal{F}$ on $U$ and any $k\geq 1$.
\end{enumerate}
\end{definition}

\begin{remark}
The letter "c" in the term "p-good space' stands for "coherent".
\end{remark}

\begin{definition}\label{defi: c-good covers}
Let $(X,\mathcal{O}_X)$ be a locally ringed space, an open cover $\{U_i\}$ of $X$ is called a \emph{c-good cover}   if $(U_I,\mathcal{O}_X|_{U_I})$ is a c-good space for   any finite intersection $U_I$ of the open cover.
\end{definition}

For a separated, Noetherian scheme $(X,\mathcal{O}_X)$, any affine cover $\{U_i\}$ is c-good, see Appendix \ref{appendix: good covers}.

Then we have the coherent version of twisted resolution (Proposition \ref{prop: existence of resolution}).
\begin{proposition}\label{prop: existence of resolution for coherent complexes}
Assume the cover $\{U_i\}$ is   c-good. Let $P=(S^{\bullet},d_S)$ be a bounded above coherent complex which consists of quasi-coherent modules, then there exists a  twisted coherent complex  $\mathcal{E}$ together with a weak equivalence
$$
\phi: \mathcal{E}\overset{\sim}{\to} \mathcal{T}(P).
$$
\end{proposition}
\begin{proof}
The proof is the same as that of Proposition \ref{prop: existence of resolution}.
\qed \end{proof}

Hence we have the following essential surjectivity.
\begin{corollary}\label{coro: essentially surjective coherent complexes}
If the cover $\{U_i\}$ is c-good, then the sheafification functor
$$
\widetilde{\mathcal{S}}: \text{Tw}^-_{\text{coh}}(X)\to \text{Qcoh}^-_{\text{coh}}(X)
$$
induces an essentially surjective functor
$$
\widetilde{\mathcal{S}}: Ho\text{Tw}^-_{\text{coh}}(X)\to D^-_{\text{coh}}(\text{Qcoh}(X)).
$$
\end{corollary}
\begin{proof}
The proof is the same as that of Corollary \ref{coro: essentially surjective}.
\qed \end{proof}

The essential surjectivity on arbitrary $\mathcal{O}_X$-modules involves the following definition.
\begin{definition}\label{defi: coherent-equivalent condition}
We say a locally ringed space $(X,\mathcal{O}_X)$ satisfies the \emph{coherent-equivalent condition} if the natural map
$$
D^-_{\text{coh}}(\text{Qcoh}(X))\to D^-_{\text{coh}}(X)
$$
is an equivalence.
\end{definition}

Actually we can show that any Noetherian scheme with finite Krull dimension satisfies the coherent-equivalent condition, see Appendix \ref{appendix: quasi-coherent, pseudo-coherent complexes and coherent sheaves} Corollary \ref{coro: spaces satisfy the coherent-equivalent condition}.

\begin{corollary} \label{coro: essentially surjective coherent complexes arbitrary modules}
If $X$ satisfies the coherent-equivalent  condition and the cover $\{U_i\}$ is c-good, then the functor
$$
\widetilde{\mathcal{S}}: Ho\text{Tw}^-_{\text{coh}}(X)\to D^-_{\text{coh}}(X)
$$
is essentially surjective.
\end{corollary}
\begin{proof}
It is obvious from Corollary \ref{coro: essentially surjective coherent complexes} and Definition \ref{defi: coherent-equivalent condition}.
\qed \end{proof}

\subsubsection{The fully-faithfulness on coherent complexes}
\begin{proposition}\label{prop: invertible in hoTw and quasi-isomorphism for coherent}
Let the cover $\{U_i\}$ satisfy $H^k(U_i,\mathcal{O}_X|_{U_i})=0$ for any $i$ and any $k\geq 0$.  If $\mathcal{E}$ and $\mathcal{F}$ are both in the subcategory $\text{Tw}^-_{\text{coh}}(X)$, then $\widetilde{\mathcal{S}}(\phi): \widetilde{\mathcal{S}}(E)\to \widetilde{\mathcal{S}}(\mathcal{F})$ is a quasi-isomorphism if and only if    $\phi: \mathcal{E}\to \mathcal{F}$ is invertible in $Ho\text{Tw}^-_{\text{coh}}(X)$.
\end{proposition}
\begin{proof}
Since we have Proposition \ref{prop: homotopy invertible morphisms for coherent}, the proof is the same as that of Proposition \ref{prop: invertible in hoTw and quasi-isomorphism} in Section \ref{section: twisted complex and dg-enhancement}.
\qed \end{proof}

\begin{corollary}\label{coro: fully faithful of coherent complexes}
If  the cover $\{U_i\}$ is c-good, then the functor $\widetilde{\mathcal{S}}: Ho\text{Tw}^-_{\text{coh}}(X)\to D^-_{\text{coh}}(\text{Qcoh}(X))$ is fully faithful.
\end{corollary}
\begin{proof}
The proof is the same as that of Corollary \ref{coro: fully faithful}.
\qed \end{proof}

\begin{theorem}\label{thm: dg-enhancement of coherent}
If the cover $\{U_i\}$ is c-good, then the sheafification functor $\widetilde{\mathcal{S}}: \text{Tw}^-_{\text{coh}}(X)\to \text{Qcoh}^-_{\text{coh}}(X)$ gives  an equivalence of categories
$$
\widetilde{\mathcal{S}}: Ho\text{Tw}^-_{\text{coh}}(X)\to D^-_{\text{coh}}(\text{Qcoh}(X))
$$
\end{theorem}
\begin{proof}
It is a immediate consequence of Corollary \ref{coro: essentially surjective coherent complexes} and \ref{coro: fully faithful of coherent complexes}.
\qed \end{proof}

\begin{example}
If $X$ is a separated Noetherian scheme and $\{U_i\}$ is an affine cover, then we have an equivalence of categories $\widetilde{\mathcal{S}}:  Ho\text{Tw}^-_{\text{coh}}(X)\to D^-_{\text{coh}}(\text{Qcoh}(X))$.
\end{example}

Then we consider the coherent complexes of arbitrary $\mathcal{O}_X$-modules.

\begin{theorem}\label{thm: dg-enhancement on coherent of O-modules}
If $X$ satisfies the coherent-equivalent  condition and the cover $\{U_i\}$ is c-good, then the sheafification functor $\widetilde{\mathcal{S}}: \text{Tw}^-_{\text{coh}}(X)\to \text{Qcoh}^-_{\text{coh}}(X)$ gives an  equivalence of categories
$$
\widetilde{\mathcal{S}}: Ho\text{Tw}^-_{\text{coh}}(X)\to D^-_{\text{coh}}(X)
$$
\end{theorem}
\begin{proof}
This is a immediate consequence of Theorem \ref{thm: dg-enhancement of coherent} and Definition \ref{defi: coherent-equivalent condition}.
\qed \end{proof}

\begin{example}
If $X$ is a separated Noetherian scheme with finite Krull dimension and $\{U_i\}$ is an affine cover, then we have an equivalence of categories $\widetilde{\mathcal{S}}: Ho\text{Tw}^-_{\text{coh}}(X)\to D^-_{\text{coh}}(X)$. See Appendix \ref{appendix: quasi-coherent, pseudo-coherent complexes and coherent sheaves} Corollary \ref{coro: spaces satisfy the coherent-equivalent condition}.
\end{example}

\subsection{Degenerate twisted complexes}\label{subsection: splitting}
Recall that in the definition of  twisted complex we have the non-degenerate condition which requires that on each $U_i$ we have
$$
a^{1,0}_{ii}=id
$$
up to homotopy.

It is interesting to see what happens if we drop the non-degenerate condition. In fact we have the following definition.

\begin{definition}\label{defi: generalized twisted and twisted complex}
A generalized twisted complex is the same as a twisted complex except that we do not require $a^{1,0}_{ii}=id$ up to homotopy.

Similarly we have generalized twisted perfect complexes and generalized twisted coherent complexes.

We denote the dg-category of generalized twisted complexes by gTw$(X)$.

Similarly we have $\text{gTw}_{\text{perf}}(X)$ and $\text{gTw}^-_{\text{coh}}(X)$.
\end{definition}

\begin{example}\label{eg: example of generalized twisted complex}
For given $E^{\bullet}_i$'s, we could set all $a^{k,1-k}$'s to be $0$. It definitely satisfies the Maurer-Cartan equation $\delta a+a\cdot a=0$ hence it gives a generalized twisted complex but not a twisted complex unless the $E^{\bullet}_i$'s are all zero.
\end{example}

For generalized twisted complexes we have the following obvious observations
\begin{enumerate}

\item Tw$(X)$ is a full dg-subcategory of gTw$(X)$, $\text{Tw}_{\text{perf}}(X)$ is a full dg-subcategory of $\text{gTw}_{\text{perf}}(X)$ and  $\text{Tw}^-_{\text{coh}}(X)$ is a full dg-subcategory of $\text{gTw}^-_{\text{coh}}(X)$.

\item Nevertheless there is no inclusion relation between $\text{gTw}_{\text{perf}}(X)$ and Tw$(X)$ nor between $\text{gTw}^-_{\text{coh}}(X)$ and Tw$(X)$.

\item The pre-triangulated structure as in Section \ref{subsection: pre-triangulated structure} can be defined on gTw$(X)$, $\text{gTw}_{\text{perf}}(X)$ and  $\text{gTw}^-_{\text{coh}}(X)$ without any change.

\item The weak equivalence in gTw$(X)$ is exactly the same as in Section \ref{subsection: weak equivalence}. Moreover Definition \ref{defi: quasi-isomorphism in Tw} and Proposition \ref{prop: homotopy invertible morphisms} still hold for generalized twisted complexes.

\item   We can define the sheafification functor
$$
\mathcal{S}: \text{gTw}(X)\to \text{Sh}(X)
$$
in the same way as Section \ref{subsection: sheafification functor} Definition \ref{defi: sheaf associated to twisted complex} and \ref{defi: sheafification functor}.
\end{enumerate}

It is not obvious that $\mathcal{S}$ maps a generalized twisted perfect/coherent complex to a perfect/coherent complex. Actually we need some more work. Recall Lemma \ref{lemma: aii is an idempotent} claims that if the $a^{k,1-k}$'s satisfy the Maurer-Cartan equation, then $ a^{1,0}_{ii}: ( E^n_i, a^{0,1}_i)\to (E^n_i, a^{0,1}_i)$ is an idempotent map in the homotopy category $K(U_i)$, i.e. $(a^{1,0}_{ii})^2=a^{1,0}_{ii}$ up to chain homotopy.

It is a classical result that the category $K(U_i)$ is \emph{idempotent complete} (\cite{bokstedt1993homotopy} Proposition 3.2), i.e. for any object $S$ of  $K(U_i)$ and any idempotent $\alpha: S\to S$, there exists a splitting of $\alpha$. More precisely there exists a $T$  in $K(U_i)$ together with $i: T\to S$ and $p: S\to T$ such that
$$
pi=id_T \text{ and } ip=\alpha.
$$
Intuitively such a splitting $T$ can be considered as the image of the map $\alpha$. However in general $T$ is not the naive image of $\alpha$ in the chain complex.

The following proposition gives an explicit construction of the splitting.

\begin{proposition}\label{prop: associated sheaf splitting}
Let $\mathcal{E}=(E^{\bullet}_i,a)$ be a generalized twisted complex and $(\mathcal{S}^{\bullet}(\mathcal{E}),\delta_a)$ be the associated complex of sheaves. Then $(\mathcal{S}^{\bullet}(\mathcal{E}),\delta_a)|_{U_j}$ is a splitting of the idempotent $ a^{1,0}_{jj}: ( E^n_j, a^{0,1}_j)\to (E^n_j, a^{0,1}_j)$, i.e. we have two morphisms
$$
f: (\mathcal{S}^{\bullet}(\mathcal{E}),\delta_a)|_{U_j} \to (E^{\bullet}_{j},a^{0,1}_{j})
$$
and
$$
g: (E^{\bullet}_{j},a^{0,1}_{j}) \to (\mathcal{S}^{\bullet}(\mathcal{E}),\delta_a)|_{U_j}
$$
such that
$$
f\circ g=a^{1,0}_{jj} \text{ and } g\circ f=id_{\mathcal{S}^{\bullet}(\mathcal{E})|_{U_j}} \text{ up to chain homotopy.}
$$
\end{proposition}
\begin{proof}
The proof is exactly the same as that of Proposition \ref{prop: associated sheaf locall isom to original twisted complex} except that here $f\circ g=a^{1,0}_{jj}$ does not necessarily equal to $id$, not even up to homotopy.
\qed \end{proof}

With the help of Proposition \ref{prop: associated sheaf splitting} we can get the following result.

\begin{corollary}\label{coro: perfectness of sheafification generalized twisted complex}
If $\mathcal{E}=(E^{\bullet},a)$ is a generalized twisted perfect (or twisted coherent) complex, then the sheafification $\mathcal{S}^{\bullet}(\mathcal{E})$ is a perfect (or coherent, respectively) complex of sheaves on $(X,\mathcal{O}_X)$. In other words the sheafification functor $\mathcal{S}$ restricts to $\text{gTw}_{\text{perf}}(X)$ (or $ \text{gTw}^-_{\text{coh}}(X)$, respectively) and gives the following dg-functor
$$
\mathcal{S}: \text{gTw}_{\text{perf}}(X)\to \text{Qcoh}_{\text{perf}}(X).
$$
and
$$
\mathcal{S}: \text{gTw}^-_{\text{coh}}(X)\to \text{Qcoh}_{\text{coh}}(X).
$$
\end{corollary}
\begin{proof}
Since $\mathcal{E}=(E^{\bullet},a)$ is a generalized twisted perfect complex, for each $U_j$ the complex $(E^{\bullet}_j,a^{0,1}_j)$ is a two-side bounded complex which consists of locally free finitely generated $\mathcal{O}_X$-modules, i.e. $(E^{\bullet}_j,a^{0,1}_j)$ is an object in $K_{\text{perf}}(U_j)$. We know that $K_{\text{perf}}(U_j)$ is also idempotent complete since it consists of compact objects in $K(U_j)$.  Proposition \ref{prop: associated sheaf splitting} tells us that $\mathcal{S}^{\bullet}(\mathcal{E})|_{U_j}$ is a splitting of idempotent $a^{1,0}_{jj}$ hence $\mathcal{S}^{\bullet}(\mathcal{E})|_{U_j}$ is perfect on $U_j$ Moreover this is true for any member $U_j$ of the open cover, therefore $\mathcal{S}^{\bullet}(\mathcal{E})$ is a perfect complex of sheaves on $(X,\mathcal{O}_X)$.

The same proof works for twisted coherent complexes.
\qed \end{proof}

\begin{corollary}\label{coro: essentially surjective generalized twisted complex}
\begin{enumerate}[a.]
\item If the cover $\{U_i\}$ is p-good, then the functor
$$
\mathcal{S}: Ho(\text{gTw}_{\text{perf}}(X))\to D_{\text{perf}}(\text{Qcoh}(X))
$$
is essentially surjective.

\item If the cover $\{U_i\}$ is c-good, then the functor
$$
\mathcal{S}: Ho(\text{gTw}^-_{\text{coh}}(X))\to D_{\text{coh}}(\text{Qcoh}(X))
$$
is essentially surjective.
\end{enumerate}
\end{corollary}
\begin{proof}
By Corollary \ref{coro: essentially surjective} we already know that $\mathcal{S}: Ho\text{Tw}_{\text{perf}}(X)\to D_{\text{perf}}(\text{Qcoh}(X))$  is essentially surjective. Since $\text{Tw}_{\text{perf}}(X)$ is a subcategory of $\text{gTw}_{\text{perf}}(X)$ and the functors $\mathcal{S}$'s coincide on $\text{Tw}_{\text{perf}}(X)$, the claim is obviously true.

The same proof works for twisted coherent complexes.
\qed \end{proof}

However,  $\mathcal{S}$ does not induce a fully faithful functor
$$
\mathcal{S}: Ho(\text{gTw}_{\text{perf}}(X))\to D_{\text{perf}}(\text{Qcoh}(X))
$$
nor
$$
\mathcal{S}: Ho(\text{gTw}^-_{\text{coh}}(X))\to D_{\text{coh}}(\text{Qcoh}(X)).
$$

The main reason of the failure is that we no longer have the same result as in Corollary \ref{coro: weakly equi and sheafification} for generalized twisted complexes and Proposition \ref{prop: homotopy invertible morphisms} does not hold for generalized  twisted complexes either.

In fact, if $\mathcal{E}$ and $\mathcal{F}$ are generalized twisted coherent complexes, then the fact that $\mathcal{S}(\phi): \mathcal{S}(\mathcal{E})\to \mathcal{S}(\mathcal{F})$ is a quasi-isomorphism does not imply $\phi: \mathcal{E}\to \mathcal{F}$ is invertible in the homotopy category.

\begin{example}\label{eg: quasi-isom not homotopy invertible}
For a counter-example, let $\mathcal{E}=(E^{\bullet}_i,0)$ be non-zero, two-side bounded graded locally free finitely generated $\mathcal{O}_X$-modules on each $U_i$ with all $a$'s equal to $0$. Let $\mathcal{F}$ simply be $0$ and $\phi$ be the zero map. It is clear that $\phi^{0,0}_i: E^{\bullet}_i\to 0$ is not a quasi-isomorphism hence $\phi$ cannot be invertible in $Ho(\text{gTw}_{\text{perf}}(X))$. However by Proposition \ref{prop: associated sheaf splitting} it is not difficult to show that $\mathcal{S}(\mathcal{E})$ is an acyclic complex hence $\mathcal{S}(\phi)=0: \mathcal{S}(\mathcal{E})\to 0$ is a quasi-isomorphism.
\end{example}

The above discussion tells us that $(\text{gTw}_{\text{perf}}(X),\mathcal{S})$ (or $(\text{gTw}^-_{\text{coh}}(X),\mathcal{S})$) is not a dg-enhancement of $D_{\text{perf}}(\text{Qcoh}(X))$ (or $D_{\text{coh}}(\text{Qcoh}(X))$ respectively). Nevertheless, gTw$(X)$   has its own interests and may be further studied in the future.

\subsection{Quillen adjunction}\label{subsection: quillen adjunction}
The proof of dg-enhancement in this paper is more or less a by-hand proof. Nevertheless in this section we would like to briefly mention a more categorical approach which we hope can give a systematic proof of the result in this paper.

We have defined two dg-functors
$$
\mathcal{S}: \text{Tw}(X)\to \text{Sh} (X)
$$
and
$$
\mathcal{T}:  \text{Sh} (X)\to \text{Tw}(X).
$$
We have found the relations between them in  Proposition \ref{prop: ST and id are quasi-isomorphic},  Proposition \ref{prop: TS and id are isomorphic} and Proposition \ref{prop: adjunction of T and S}.

On the other hand we have the injective and projective \emph{model structure} on Sh$(X)$, see \cite{hovey2001model}. Moreover in Definition \ref{defi: quasi-isomorphism in Tw} we already have a notion of weak equivalence in Tw$(X)$ and we wish to further construct a suitable \emph{model structure} on Tw$(X)$ with the weak equivalence as above, which, together with the suitable model structure on Sh$(X)$, makes $\mathcal{S}$ and $\mathcal{T}$ a \emph{Quillen adjunction}  and further a Quillen equivalence
$$
\mathcal{S}: \text{Tw}(X) \leftrightarrows \text{Sh} (X) :\mathcal{T}.
$$

The Quillen adjunction, if exists, will reveal deeper information on twisted complexes. It is also hoped that the dg-enhancement result can be also proved in this approach.


\appendix
\appendixpage
\addappheadtotoc

\section{SOME DISCUSSIONS ON COMPLEXES OF SHEAVES} \label{appendix: quasi-coherent, pseudo-coherent complexes and coherent sheaves}
\subsection{Pseudo-coherent complexes and coherent complexes}
Recall that we have a definition of coherent complexes in Section \ref{subsection: twisted coherent complexes}.

\begin{definition}[Definition \ref{defi: coherent complex and strictly coherent complex}]
Let $(X,\mathcal{O}_X)$ be a separated, Noetherian scheme. A complex $\mathcal{S}^{\bullet}$ of $\mathcal{O}_X$-modules is \emph{coherent} if for any point $x\in X$, there exists an open neighborhood $U$ of $x$ and a bounded above complex of finite rank, locally free sheaves  $\mathcal{E}^{\bullet}_{U}$ on $U$ such that the restriction $\mathcal{S}^{\bullet}|_U$ is isomorphic to $\mathcal{E}^{\bullet}_{U}$ in $D(\mathcal{O}_X|_U-\text{mod})$, the derived category of sheaves of $\mathcal{O}_X$-modules on $U$.
\end{definition}

For general locally ringed spaces $(X,\mathcal{O}_X)$, this version of coherent complex does not behave well and we have the following  definition.

\begin{definition}\label{defi: pseudo-coherent complex}[\cite{thomason1990higher} Definition 2.1.1, 2.2.6 or \cite{berthelot1966seminaire} Expos\'{e} I, 2.1, 2.3]
\begin{enumerate}[a.]
\item For an integer $m$, a complex $\mathcal{E}^{\bullet}$ of $\mathcal{O}_X$-modules on $X$ is called \emph{strictly $m$-pseudo-coherent} if $\mathcal{E}^i$ is a locally free finitely generated $\mathcal{O}_X$-module for $i\geq m$ and $\mathcal{E}^i=0$ for $i$ sufficiently large.

\item A complex $\mathcal{E}^{\bullet}$ of $\mathcal{O}_X$-modules on $X$ is called \emph{strictly  pseudo-coherent} if it is $m$-strictly-pseudo-coherent for all $m$, i.e. it is a bounded above complex of locally free finitely generated $\mathcal{O}_X$-modules.

\item For any integer $m$, a complex $\mathcal{E}^{\bullet}$ of $\mathcal{O}_X$-modules on $X$ is called m\emph{-pseudo-coherent} if for any point $x\in X$ there exists an open neighborhood $x\in U \subset X$ and a morphism of complexes
$\alpha  : \mathcal{P}_{U}^{\bullet} \to \mathcal{E}^{\bullet}|_{U}$
where $\mathcal{P}_U$ is strictly $m$-pseudo-coherent on $U$ and
$\alpha$ is a quasi-isomorphism on $U$.

\item We say $\mathcal{E}^\bullet$ is \emph{pseudo-coherent}
if it is $m$-pseudo-coherent for all $m$.
\end{enumerate}
\end{definition}

We may hope that a pseudo-coherent complex is locally quasi-isomorphic to a strictly pseudo-coherent complex. However according to \cite{thomason1990higher} 2.2.7 it is not true in general:
\begin{quote}
For a pseudo-coherent complex of general $\mathcal{O}_X$-modules, there will locally be $n$-quasi-isomorphisms with a strictly pseudo-coherent complex, but the local neighborhoods where the $n$-quasi-isomorphisms are defined may shrink as $n$ goes to $-\infty$, and so may fail to exist in the limit. So there may not be a local quasi-isomorphism with a strict pseudo-coherent complex.
\end{quote}

As a result, the definition of pseudo-coherent complex and our definition of coherent complex are not equivalent in general. Nevertheless if we assume  $X$ is a Noetherian scheme, then we have the following proposition.

\begin{proposition}[\cite{thomason1990higher} 2.2.8, \cite{berthelot1966seminaire} Expos\'{e} I Section 3]
A complex $E^{\bullet}$ of $\mathcal{O}_X$-modules on a Noetherian scheme $X$ is pseudo-coherent if and only if $E^{\bullet}$  is cohomologically bounded above and all the $H^k(E^{\bullet})$ are coherent $\mathcal{O}_X$-modules, i.e. $E^{\bullet}$ is pseudo-coherent if and only if $E^{\bullet}\in D^-_{\text{coh}}(X)$.
\end{proposition}
\begin{proof}
See \cite{thomason1990higher} 2.2.8 or \cite{berthelot1966seminaire} Expos\'{e} I Section 3.
\qed \end{proof}

\subsection{Quasi-coherent modules v.s. arbitrary $\mathcal{O}_X$-modules}
It is a subtle but important question  whether we could  replace a complex of $\mathcal{O}_X$-modules by a complex of quasi-coherent modules in the derived categories. In this subsection we collect some results on this topic which can be found in \cite{thomason1990higher} Appendix B and \cite{berthelot1966seminaire} Expos\'{e} II.

\begin{definition}\label{defi: quasi-coherent modules in appendix}
Let $(X,\mathcal{O}_X)$ be a locally ringed space.  A sheaf of $\mathcal{O}_X$-modules $\mathcal{F}$ is called \emph{quasi-coherent} if for every point $x\in X$ there exists an open neighbourhood $x\in U\subset X$ such that $\mathcal{F}|_U$ is isomorphic to the cokernel of a map
$$
\bigoplus_{j\in J} \mathcal{O}_X|_U\to \bigoplus_{i\in I} \mathcal{O}_X|_U.
$$
\end{definition}

\begin{remark}\label{remark: Frechet q-c sheaf}
If $(X,\mathcal{O}_X)$ is a complex manifold, then we need the category of  Fr\'{e}chet quasi-coherent sheaves, which is a variation of the category of quasi-coherent sheaves, see \cite{eschmeier1996spectral} Section 4.3 for more details.
\end{remark}

The natural inclusion $i: \text{Qcoh}(X)\to \text{Sh}(X)$ induces a natural functor
$$
\tilde{i}: D(\text{Qcoh}(X))\to D_{\text{Qcoh}}(X)
$$
where $D_{\text{Qcoh}}(X)$ is the derived category of complexes of $\mathcal{O}_X$-modules with quasi-coherent cohomologies. However the functor $\tilde{i}$ is not necessarily essentially surjective nor fully faithful. The same is true when we restrict to certain subcategories such as perfect complexes or coherent complexes.

Since $\tilde{i}: D(\text{Qcoh}(X))\to D_{\text{Qcoh}}(X)$ is not an equivalence in general, we need to impose some condition on the locally ringed space $(X,\mathcal{O}_X)$ for our purpose. Here are some definitions we use in this paper.

\begin{definition}\label{defi: perfect-equivalent condition in appendix}[See Definition \ref{defi: perfect-equivalent condition} and Definition \ref{defi: coherent-equivalent condition}]
\begin{enumerate}[a.]
\item We say a locally ringed space $(X,\mathcal{O}_X)$ satisfies the \emph{perfect-equivalent condition} if the functor
$$
D_{\text{perf}}(\text{Qcoh}(X))\to D_{\text{perf}}(X)
$$
is an equivalence.
\item We say a locally ringed space $(X,\mathcal{O}_X)$ satisfies the \emph{coherent-equivalent condition} if the functor
$$
D^-_{\text{coh}}(\text{Qcoh}(X))\to D^-_{\text{coh}}(X)
$$
is an equivalence.
\end{enumerate}
\end{definition}

It is important to verify for which $X$ the above condition holds. In fact we have the following result.

\begin{proposition}\label{prop: bounded below quasi-coherent and A-modules equivalent}[\cite{thomason1990higher} Proposition B.16, \cite{berthelot1966seminaire} Expos\'{e} II 3.5]
Let $X$ be either a quasi-compact
and semi-separated scheme, or else a Noetherian scheme. Then the functor
$$
\tilde{i}: D^{+}(\text{Qcoh}(X))\to D^+_{\text{Qcoh}}(X)
$$
is an equivalence, where $D^{+}(\text{Qcoh}(X))$ is the derived category of complexes of  quasi-coherent modules with  bounded below cohomologies, and $D^+_{\text{Qcoh}}(X)$ is the derived category of complexes of $\mathcal{O}_X$-modules with bounded below and quasi-coherent cohomologies.
\end{proposition}
\begin{proof}
See the proof of \cite{thomason1990higher} Proposition B.16.
\qed \end{proof}

\begin{corollary}\label{coro: spaces satisfy the perfect-equivalent condition}
Any  quasi-compact
and semi-separated or Noetherian scheme satisfies the perfect-equivalent condition.
\end{corollary}
\begin{proof}
On a quasi-compact scheme, any perfect complex has bounded below cohomology, hence by Proposition \ref{prop: bounded below quasi-coherent and A-modules equivalent} any quasi-compact scheme satisfies the perfect-equivalent condition  .
\qed \end{proof}

However a bounded above coherent complex is not necessarily bounded below hence we can no longer use Proposition \ref{prop: bounded below quasi-coherent and A-modules equivalent}. Nevertheless we have the same result under additional conditions.

\begin{proposition}\label{prop: quasi-coherent and A-modules equivalent}[\cite{thomason1990higher} B.17]
Let $X$ be either a Noetherian scheme of finite Krull dimension  or
a semi-separated scheme with underlying space a Noetherian space of
finite Krull dimension. Then the functor
$$
\tilde{i}: D(\text{Qcoh}(X))\to D_{\text{Qcoh}}(X)
$$
is an equivalence.
\end{proposition}
\begin{proof}
See   \cite{thomason1990higher}  B.17.
\qed \end{proof}

\begin{corollary}\label{coro: spaces satisfy the coherent-equivalent condition}
Any Noetherian scheme of finite Krull dimension  or
a semi-separated scheme with underlying space a Noetherian space of
finite Krull dimension satisfies the coherent-equivalent condition.
\end{corollary}
\begin{proof}
It is a direct corollary of Proposition \ref{prop: quasi-coherent and A-modules equivalent}.
\qed \end{proof}


\section{GOOD COVERS OF LOCALLY RINGED SPACES}\label{appendix: good covers}
We discuss good covers of locally ringed spaces in this appendix. Recall that we have the following definitions.

\begin{definition}\label{defi: good space in appendix}[Definition \ref{defi: good space}]
\begin{enumerate}[a.]
\item A locally ringed space $(U,\mathcal{O}_U)$ is called \emph{p-good} if it satisfies the following two conditions
\begin{enumerate}[1.]
\item For every perfect complex $\mathcal{P}^{\bullet}$ on $U$ which consists of quasi-coherent sheaves, there exists a strictly perfect complex $\mathcal{E}^{\bullet}$ together with a quasi-isomorphism $u: \mathcal{E}^{\bullet}\overset{\sim}{\to}\mathcal{P}^{\bullet}$.
\item The higher cohomologies of quasi-coherent sheaves vanish, i.e. $H^k(U,\mathcal{F})=0$ for any quasi-coherent sheaf $\mathcal{F}$ on $U$ and any $k\geq 1$.
\end{enumerate}

\item A locally ringed space $(U,\mathcal{O}_U)$ is called \emph{c-good} if the first condition above is replaced by
For every coherent complex $\mathcal{C}^{\bullet}$ on $U$ which consists of quasi-coherent sheaves, there exists a   bounded above complex of finitely generated locally free sheaves $\mathcal{E}^{\bullet}$ together with a quasi-isomorphism $v: \mathcal{E}^{\bullet}\overset{\sim}{\to}\mathcal{C}^{\bullet}$. The second condition remains the same.
\end{enumerate}
\end{definition}

\begin{definition}\label{defi: good covers in appendix}[Definition \ref{defi: good covers}]
Let $(X,\mathcal{O}_X)$ be a locally ringed space, an open cover $\{U_i\}$ of $X$ is called a \emph{p-good cover} (or  \emph{c-good cover}) if $(U_I,\mathcal{O}_X|_{U_I})$ is a p-good space (or c-good space, respectively) for any finite intersection $U_I$ of the open cover.
\end{definition}

The definition of good cover is not too restrictive since we have the following examples of ringed spaces with good covers.
\begin{itemize}
\item $(X,\mathcal{O}_X)$ is a separated scheme, then any affine cover is both p-good and c-good. In fact on a separated scheme the intersection of two affine open subsets is still affine hence Condition $2.$ in Definition \ref{defi: good space in appendix} is obviously satisfied and Condition $1.$ is proved in \cite{thomason1990higher} Proposition 2.3.1.

\item $(X,\mathcal{O}_X)$ is a complex manifold with $\mathcal{O}_X$ the sheaf of holomorphic functions. In these case a  Stein cover is both p-good and c-good. Actually on   complex manifolds  we should use the definition  of \emph{Fr\'{e}chet quasi-coherent sheaves}, which is a variation of ordinary quasi-coherent sheaves, see \cite{eschmeier1996spectral} Section 4. A  Stein manifold satisfies Condition $2.$ by Proposition 4.3.3 in \cite{eschmeier1996spectral}, and Condition $1.$ can be proved in the same way as the argument in \cite{thomason1990higher} Section 2.

\item $(X,\mathcal{O}_X)$ is a paracompact topological space with \emph{soft} structure sheaf $\mathcal{O}_X$. Then any contractible open cover is both  p-good and c-good.
\end{itemize}

\bibliography{twistedbib}{}
\bibliographystyle{plain}
\end{document}